\documentclass[11pt,a4paper,reqno]{amsart}
\usepackage{amsfonts}
\usepackage{amssymb}
\usepackage{hyperref}

\numberwithin{equation}{section}
\newtheorem{theorem}{Theorem}[section]
\newtheorem{lemma}[theorem]{Lemma}
\newtheorem{remark}[theorem]{Remark}
\newtheorem{corollary}[theorem]{Corollary}
\newtheorem{proposition}[theorem]{Proposition}
\newtheorem{definition}[theorem]{Definition}
\newtheorem{example}[theorem]{Example}
\def\eqref#1{(\ref{#1})}
\newenvironment{acknowledgment}{\smallskip{\sc Acknowledgments.}\rm}{\smallskip}
\newenvironment{notation}{\smallskip{\sc Notation.}\rm}{\smallskip}
\allowdisplaybreaks
\def\RM{\rm}

\def\func#1{\mathop{\mathrm{#1}}\nolimits}

\def\enddoc{\end{document}}

\def\FRAME#1#2#3#4#5#6#7#8
{
 \begin{figure}[ptbh]
 \begin{center}
 \includegraphics[height=#3]{#7}
 \caption{#5}
 \label{#6}
 \end{center}
 \end{figure}
}

\begin{document}
\title[Parabolic equation]{On nonexistence and existence of positive global solutions to heat equation with a potential term on Riemannian manifolds}

\author[Gu]{Qingsong Gu}
\address{Department of Mathematics and Statistics, Memorial University of Newfoundland, A1C 5S7, NL, Canada.}
\email{001gqs@163.com}

\author[Sun]{Yuhua Sun}
%\thanks{$^{\ddag }$Supported by ***** Nankai University}
\address{School of Mathematical Sciences and LPMC, Nankai University, 300071
Tianjin, P. R. China}
\email{sunyuhua@nankai.edu.cn}

\author[Xu]{Fanheng Xu}
\address{School of Mathematical Sciences and LPMC, Nankai University, 300071
Tianjin, P. R. China}
\email{xufanheng@mail.nankai.edu.cn}

\date{January 05, 2019}
\keywords{heat equation with potential term;
Riemannian manifolds; sharp volume growth}
\subjclass{Primary: 35J61, Secondary: 58J05}
\thanks{\noindent
Sun was supported by the National Natural Science Foundation of China (No.11501303, No.11871296,
No.11761131002),
 and
also by the Fundamental Research Funds for the Central Universities.
}

\begin{abstract}
We reinvestigate nonexistence and existence of global positive solutions to heat equation with
a potential term
on Riemannian manifolds.  Especially, we give a very natural sharp condition only in terms of the volume of geodesic ball to obtain nonexistence results.
\end{abstract}
 \maketitle

\section{Introduction}

In this paper we investigate nonexistence and existence of global positive solutions to the
following problem
\begin{equation}\label{eq-pra}
\left\{
\begin{array}{ll}
{{\partial_t u} = \Delta u -V(x)u+ {u^p}} \quad \mbox{in $M \times (0, \infty )$}, \\
{u(x,0)= {u_0(x)}} \quad  \mbox{in $M$},
\end{array}
\right.
\end{equation}
where $p>1$, and $M$ is
 a connected non-compact
geodesically complete Riemannian manifold with $dimM\geq3$,
$\Delta $ is the Laplace-Beltrami operator on $M$, $V(x)$ is a smooth  function and can be allowed to
be negative,
and ${u_0}$ is a nonnegative function which is not identically zero.

The main objective of this paper is to illustrate the following questions:
\begin{enumerate}
\item[1.]{What are the influences of potential $V$ and $p$ on the
nonexistence and existence of global positive solutions to problem (\ref{eq-pra})?}
\item[2.]{ Are these influences of $p$ sharp in some kind of sense
for different potential $V$?}
\end{enumerate}

Before answering these questions, let us firstly recall some history in this area.
When $M=\mathbb{R}^N$, problem (\ref{eq-pra}) and its variations have been investigated widely in different respects, see \cite{Ban-Lev, Ban-Tesei11, Gal-Lev98, Weissler80, Weissler81}, and also a very good survey paper by Levine \cite{Lev90}.

Among these literatures,
the first celebrated result on problem (\ref{eq-pra}) is due to Fujita's famous paper \cite{Fujita} dealing
with the case when $M=\mathbb{R}^N$ and $V(x)\equiv0$.
He proved that
\begin{enumerate}
\item[(1)]{If $1<p<1+\frac{2}{N}$, and $u_0>0$, then (\ref{eq-pra}) possesses no global positive solution.}
\item[(2)]{If $p>1+\frac{2}{N}$, and $u_0$
is smaller than a small Gaussian, then (\ref{eq-pra}) has global solutions.}
\end{enumerate}
Here the number $1+\frac{2}{N}$ is called the Fujita exponent, and usually denoted
by $p^{*}$. The question of whether $p^{*}=1+\frac{2}{N}$
 belongs to the blow-up case is much more difficult.
 The case $p=1+\frac{2}{N}$ was decided by Hayakawa \cite{H} for $N=1,2$ and by Kobayashi, Sirao and Tanaka \cite{KST} for general $N$. One can also see the papers
 \cite{AW},\cite{Weissler81} for different methods and further developments.

 Zhang investigated problem (\ref{eq-pra}) when
$V(x)$ has the asymptotic behavior like
$\frac{\omega}{1+|x|^b}$
for some $\omega\neq0$ and $b>0$.  He showed that
\begin{theorem}\label{thm-Zhang}\cite[Zhang]{Zhang01}\RM\;
Let $M=\mathbb{R}^N$ with $N\geq3$.
\begin{enumerate}
\item[(1)]{If, for some $b>2$ and $\omega>0$, $0\leq V(x)\leq \frac{\omega}{1+|x|^b}$ holds, then
$p^{*}=1+\frac{2}{N}$;}
\item[(2)]{If, for some $b\in(0,2)$ and $\omega>0$, $V(x)\geq\frac{\omega}{1+|x|^b}$ holds, then $p^{*}=1$ and there exists global solutions for all $p>1$;}
\item[(3)]{If, for some $b>2$ and $\omega<0$ with
$|\omega|$ small enough, $\frac{\omega}{1+|x|^b}\leq V(x)\leq0$ holds, then $p^{*}=1+\frac{2}{N}$; }
\item[(4)]{If, for some $b\in (0, 2)$ and $\omega<0$, $V(x)\leq \frac{\omega}{1+|x|^b}$ holds,
then $p^{*}=\infty$, which means there exist no
global solutions to (\ref{eq-pra}) for any $p>1$.}
\end{enumerate}
\end{theorem}

When $V(x)$ behaves like $\frac{\omega}{|x|^2}$ for some $\omega>0$ and large $|x|$,  Ishige proved that
\begin{theorem}\label{thm-Ishige}\cite[Ishige]{Ishi08}\RM\;
Let $M=\mathbb{R}^N$ with $N\geq3$. Assume that $V(x)\geq0$. Let $\omega>0$.
\begin{enumerate}
\item[(1)]{If $V(x)\geq\frac{\omega}{|x|^2}$ for large $x$, then for $p>p^{*}(\omega)$, there exists global positive
solution to (\ref{eq-pra});}
\item[(2)]{If $V(x)\leq \frac{\omega}{|x|^2}$ for large $x$, then for $1<p\leq p^{*}(\omega)$, there exists no global positive solution to (\ref{eq-pra});}
\end{enumerate}
where
\begin{equation}
p^{*}(\omega)=1+\frac{2}{N+\alpha(\omega)},
\end{equation}
and
\begin{equation}\label{root-1}
\alpha(\omega)=\frac{-(N-2)+\sqrt{(N-2)^2+4\omega}}{2}
\end{equation}
is  the larger root of the equation $\alpha(\alpha+N-2)=\omega$.
\end{theorem}

 When $V(x)$ behaves like $\frac{\omega}{|x|^2}$, and $\omega$ can be allowed to be negative satisfying $-\frac{(N-2)^2}{4}\leq\omega<0$, Pinsky obtained that
\begin{theorem}\label{thm-Pinsky}\cite[Pinsky]{Pinsky09}\RM\;
Let $M=\mathbb{R}^N$ with $N\geq3$.
\begin{enumerate}
\item[(1)]{If $V(x)\geq \frac{\omega}{|x|^2}$, then there exists global solution to (\ref{eq-pra})
when $p>p^{*}(\omega)$;}
\item[(2)]{If $V(x)\leq\frac{\omega}{|x|^2}$, for large $|x|$, then there are no global solutions
to (\ref{eq-pra}) when $1<p\leq p^{*}(\omega)$.}
\end{enumerate}
\end{theorem}

Now let us  transfer our attentions from Euclidean space to
manifold. We make a rough assumption on manifold: assume that $M$ is a connected non-compact
geodesically complete Riemannian manifold, $d$ is the geodesic distance on $M$, and $\mu_0$ is the Riemannian
measure of $M$. Fix a reference point $x_0\in M$, let $B(x_0, r)$ denote the geodesic ball
on $M$ centered at $x_0$ with radii $r>0$.

The study of nonlinear parabolic equations on manifolds become more and more
intriguing, not only because that it has so many applications in geometry and many other areas,
but also because usually the approach which is applied for the manifold case is quite different from the Euclidean ones.

In \cite{Zhang-duk99},  Zhang provided a unified approach to obtain
blow-up results for several variations of problem (\ref{eq-pra}) when $V(x)=0$.
To cite his result more precisely,  let us introduce his  assumptions on the manifold
\begin{enumerate}
\item[(i).]{$\mu_0(B(x, r))\leq Cr^{\alpha}$, when $r$ is large and for all $x\in M$.}
\item[(ii).]{$\frac{\partial \log g^{\frac{1}{2}}}{\partial r}\leq \frac{C}{r}$, where $r=d(x_0, x)$ is smooth. Here $x_0$ is a fixed reference point, and $g^{\frac{1}{2}}$ is the volume density of the manifold.}
\end{enumerate}
 Zhang obtained  Fujita exponent of  problem (\ref{eq-pra}) when $V(x)=0$.
\begin{theorem}\label{zhang-fujita}\cite[Zhang]{Zhang-duk99}\RM\;
Assume  conditions (i) and (ii) on manifold are satisfied, and $\alpha\geq1$. If $1<p\leq 1+\frac{2}{\alpha}$, then problem (\ref{eq-pra})
possesses no global positive solution to (\ref{eq-pra}).
\end{theorem}
The approach applied by Zhang in \cite{Zhang-duk99}  is quite powerful, and even very effective to  nonlinear homogeneous and
inhomogeneous equations, semilinear parabolic equations  and porous medium equations with nonlinear source, even to the blow-up problems in exterior domains \cite{Zhang-ext}.
Zhang's approach is by first constructing a suitable integral functional to show that the integral functional in selected
fixed domain will blow-up or will be identically equal to zero, then one can derive the blow-up results of nonlinear parabolic equations on manifolds.
However, after a very careful examination of Zhang's paper \cite{Zhang-duk99}, one can find that the assumptions  (i) and (ii) on manifold are essential in his approach, either can not
be relaxed or can not be dropped,
and also, the paper \cite{Zhang-duk99}
needs to deal with the critical case in a separate way to obtain the blow-up results.

In \cite{MMP-P}, Mastrolia, Monticelli and Punzo investigated the  problem (\ref{eq-pra})
with $V(x)\equiv0$
\begin{equation}\label{eq:para-ineq}
\left\{
\begin{array}{ll}
\partial_t u =\Delta u + u^p \quad \mbox{in $M \times (0, \infty )$}, \\
{u(x,0)= {u_0(x)}} \quad  \mbox{in $M$},
\end{array}
\right.
\end{equation}
They showed that Zhang's result can be improved: assumption (ii) can be dropped and assumption (i) can be relaxed to
a milder version
\begin{equation}\label{vol-mmp}
\mu_0(B(x_0, r))\leq Cr^{\alpha}\ln^{\frac{\alpha}{2}}r,\quad\mbox{for large enough $r$},
\end{equation}
for some reference point $x_0$, the same result still holds.
Their technique is to multiply the equation (\ref{eq-pra}) by $u^{a}\varphi^b $, and to obtain an
 integral estimate involving $u$ to show the nonexistence results.
 This technique is called the nonlinear capacity method,
which is systematically studied by Mitidieri and Pohozaev to deal with the elliptic inequality and parabolic differential inequalities.
Let us refer to \cite{Miti08, Miti09, Miti98, Miti01} for more details.

Here we point out that their proof relies on a very delicate choice of test function $\varphi$. Moreover, the sharpness of $\frac{\alpha}{2}$ is not shown in their paper \cite{MMP-P}.

In this paper, the purpose of the paper is threefold:
 the first one is to provide a sufficient condition for
the nonexistence of global solution to problem (\ref{eq-pra}) with  general $V$; the second one is to
attempt to
show a unified approach to deal with the parabolic equation with the potential term, moreover,
we present a totally different test function $\varphi$ from the one used in \cite{MMP-P}; the third one is to
show the sharpness of the (general) volume assumption of $\frac{\alpha}{2}$, which has not been shown before.

The idea of using the upper bound of volume of geodesic ball to derive  Liouville's
uniqueness type result has already been widely used
in literature. It originated from the celebrated work of Cheng and Yau \cite{ChengYau}. They proved that if
 on a geodesically complete
Riemannian manifold $M$, for some reference point $x_0\in M$, the following
\begin{equation*}
\mu_0(B(x_0, r))\leq Cr^{2},  \label{r2}
\end{equation*}%
holds for all large enough $r$, then any non-negative superharmonic function on $M$ is identically constant.
For other
related studies in this area we refer the readers to \cite{Grigoryan85,
Grigoryan-book,MMP-P, Wang-Xiao}.

Our paper is inspired by the elliptic results in \cite{Grigoryan13}, \cite{Gri-Sun-Igor18} and \cite{Sun-jmaa}, and
parabolic results in \cite{MMP-P}.
In the paper \cite{Grigoryan13}, Grigor'yan and the second author investigated the following differential inequality on $M$
\begin{equation}
\Delta u+u^{\sigma }\leq0,  \label{equ-const}
\end{equation}%
and proved that if, for some reference point $x_0\in M$ and $\alpha >2$, the following
\begin{equation}
\mu_0(B(x_0, r))\leq Cr^{\alpha }\ln ^{\frac{\alpha -2}{2}}r,  \label{Vra}
\end{equation}%
holds for all large enough $r$, then, for any $\sigma \leq \frac{\alpha }{\alpha -2}$,
the only nonnegative solution to (\ref{equ-const}) is identically equal to zero.
They also showed the exponents $\alpha$ and $\frac{\alpha-2}{2}$
in (\ref{Vra}) are sharp, and can not be relaxed.
Otherwise, there exists some model manifold which satisfies (\ref{Vra}) and admits positive solution
to (\ref{equ-const}).
The main technique applied in \cite{Grigoryan13} relies on a very delicate choice of test function on manifolds.

Recently in \cite{Gri-Sun-Igor18}, Grigor'yan, the second author and Verbitsky generalized the above results to the integrated
form, they obtained the necessary and sufficient condition for the existence of positive solutions
in terms of Green function of $\Delta$. Especially, when $M$ has nonnegative Ricci curvature, they showed that  problem (\ref{Vra}) admits a positive $C^2\text{-}$solution
if and only if
\begin{equation}
\int_{r_0}^{\infty}\frac{r^{\sigma-1}}{[\mu_0(B(x_0,r))]^{\sigma-1}}dr<\infty,
\end{equation}
for some reference point $x_0$ and $r_0>0$.

Further in \cite{Sun-jmaa}, the second author used two different test functions to show
that if the volume of geodesic ball satisfies some suitable growth, then the uniqueness result
of nonnegative solutions for semi-linear elliptic differential inequalities holds.

%To be specific,
%we completely remove the assumption (ii) on manifolds, and relax he assumption (i)
%to a milder version
%$$\mu_0(B(x_0, r))\leq Cr^{\alpha}\ln^{\frac{\alpha}{2}}r,\quad\mbox{for large enough $r$},$$
%for some reference point $x_0$. We emphasize here that this volume growth condition is sharp for the
%blow-up results, which means that if the volume
% of geodesic ball exceeds $Cr^{\alpha}\ln^{\frac{\alpha}{2}}r$, then the manifold may admit a global positive solution.
%Here we enjoy the duty to point that the relaxed volume assumption is much larger
%than the former one used by Zhang in \cite{Zhang-duk99},
%and we do not need to require that the volume assumption holds for every point in manifold.
%On the other hand, we also show that our approach does not need to deal with critical case separately.

Throughout the paper, we require that $V$ admits a smooth positive solution to
\begin{equation}
\Delta h=Vh,
\end{equation}
on $M$.
Actually, such a solution $h$ exists widely, for example,
\begin{lemma}
\label{h-exist}\cite[Lemmas 10.1 and 10.3]{Grigoryan06}\RM\; For any smooth
non-negative function $\Psi $ on $M$, there exists a smooth positive
function $h$ such that
\begin{equation}
\Delta h=\Psi h\quad \text{on }M.  \label{h-def}
\end{equation}%
If in addition $\Psi $ is Green bounded, namely,
\begin{equation}
\sup_{x\in M}\int_{M}G(x,y)\Psi(y)d\mu_0(y)<\infty,
\end{equation}
then the equation \emph{(\ref{h-def}%
)} has a solution $h\asymp1$ on $M$.
Here, $G(x,y)$ is a finite positive Green function with respect to $\Delta$ on $M$, and
the sign $\asymp$ means the ratio of the left-hand and right-hand is bounded from above and below
by two positive constants.
\end{lemma}

We then apply the technique of Doob's $h\text{-}$transform. Consider the weighted manifold $(M,\mu )$,
 where $\mu $ is a measure on $M$ defined by
 \begin{equation}\label{meas-mu}
d\mu:=h^2d\mu_0.
 \end{equation}
  The weighted Laplacian $\tilde{\Delta} $ of $(M,\mu )$ is defined by
\begin{equation*}
\tilde{\Delta}:  =\frac{1}{h^2 }\func{div}(h^2 \nabla).
\end{equation*}%
In particular, if $h \equiv 1$ then $\tilde{\Delta}$ is the Laplace-Beltrami
operator $\Delta$ on $M$.

By using $\Delta h=Vh$, for any smooth function $v(x)$, we know
\begin{equation*}
\tilde\Delta v+V v=(\Delta v+2\frac{\nabla h\cdot\nabla v}{h} )+\frac{\Delta h}{h}v
=\frac{1}{h}(h\Delta v+2\nabla h\cdot\nabla v+\Delta h v)=\frac{\Delta(hv)}{h},
\end{equation*}
Whence
\begin{equation*}
\tilde \Delta v=\frac{1}{h}(\Delta (hv)-V hv),
\end{equation*}
and
\begin{equation*}
\tilde\Delta=\frac{1}{h}\circ(\Delta-V)\circ h.
\end{equation*}

Let $u$ be a smooth positive solution to (\ref{eq-pra}) and let $u=hv$, we know from the above $v$
is a smooth positive global solution to the following Cauchy problem
\begin{equation}\label{eqT}
\left\{
\begin{array}{ll}
{{\partial_t v} = \tilde\Delta v+ {h^{p-1}v^p}} \quad \mbox{in $M \times (0, \infty )$}, \\
{v(x,0)= {v_0(x)}} \quad  \mbox{in $M$},
\end{array}
\right.
\end{equation}
where $v_0(x)=\frac{u_0}{h}(x)$.
Conversely, if $v$ is a smooth positive solution to problem (\ref{eqT}), then
$u=hv$ is a solution to (\ref{eq-pra}) with $u_0=hv_0$. Hence, the two problems
(\ref{eq-pra}) and (\ref{eqT}) are equivalent in the classical sense so that we only need to deal with
(\ref{eqT}) in the following. Actually, problems (\ref{eq-pra})
and (\ref{eqT}) can also be seen equivalent from the weak sense in the below.

%Solutions of (\ref{eq-pra}) and (\ref{eqT}) are understood {\color{blue}in} the weak sense.
 Denote by $W_{loc}^{1, 2}\left( M, d\mu\right)$ the space of functions $f\in L_{loc}^{2}\left(
M, d\mu\right) $  whose weak gradient $\nabla f$ is also in $L_{loc}^{2}\left(
M, d\mu\right)$. Denote by $W_{c}^{1, 2}\left( M, d\mu\right) $ the subspace of $
W_{loc}^{1, 2}\left( M, d\mu\right) $ of functions with compact support. Spaces
$W_{loc}^{1, 2}(M\times [0, \infty ), d\mu dt), W_{c}^{1, 2}(M\times [0, \infty ), d\mu dt)$
are defined similarly.
\begin{definition}\label{eq-def}\RM
$v$ is called a global weak solution to (\ref{eqT})
if $v$ is a nonnegative $W_{loc}^{1, 2}(M\times [0, \infty ), d\mu dt)$ function,
and for any nonnegative function $\psi \in W_{c}^{1, 2}(M\times [0, \infty ), d\mu dt)$, the following holds
\begin{eqnarray}
\label{eq:definition_of_weak_solution}
 \int_M \psi (x,0)v_0 d\mu + \int_0^\infty \int_M[ v\partial_t \psi  - (\nabla v, \nabla \psi) + h^{p-1}v^p\psi]d\mu dt = 0.
\end{eqnarray}
\end{definition}
\begin{remark}\RM
From Definition \ref{eq-def}, we know if $v$ is a weak solution to (\ref{eq-pra}), and $v_0$ is nonnegative, we obtain, for any
nonnegative function $\psi \in W_{c}^{1, 2}(M\times [0, \infty ), d\mu dt)$
\begin{eqnarray}\label{def: weak-sol-ineq}
\int_0^\infty \int_M h^{p-1}v^p\psi d\mu dt \leq \int_0^\infty \int_M (\nabla v, \nabla \psi) d\mu dt
-\int_0^{\infty}\int_{M}v\partial_t\psi d\mu dt.
\end{eqnarray}
\end{remark}

Before presenting the main results, we introduce some notations.
Let us define
\begin{eqnarray}\label{def-ab}
P := \frac{2}{{p - 1}}, \quad Q:= \frac{1}{{p - 1}},
\end{eqnarray}
and  a new measure $\nu$ on $M$ by
\begin{equation}\label{nu-meas}
d\nu=h^{-1}d\mu=hd\mu_0.
\end{equation}

We say that condition $(H)$ holds: if $\Delta h=Vh$ admits a smooth positive solution
$h$ and
there exist two nonnegative constants $\delta_1, \delta_2$, and some reference point $x_0$ such that
\begin{equation}\tag{$H$}
cr^{-\delta_1}\leq h(x)\leq Cr^{\delta_2},\quad\mbox{for large enough $r=d(x,x_0)$}.
\end{equation}

Our main result is the following.
\begin{theorem}\label{thm:1}\RM
Assume that condition $(H)$ is satisfied on $M$. If the following
\begin{eqnarray}\label{volume-1}
\nu (B(x_0, r)) \leq C{r^{P}}{\ln ^{Q}}r,
\end{eqnarray}
holds for all large enough $r$,
then  problem (\ref{eq-pra}) admits no global positive solution.
Here $P$ and $Q$ are defined as in (\ref{def-ab}).
\end{theorem}

In particular, when $V\equiv0$, we choose $h\equiv1$, and hence condition $(H)$ is satisfied. By Theorem
\ref{thm:1}, we have
\begin{corollary}\label{cor:1}\RM
For $V\equiv0$, if, for some reference point ${x_0}\in M$, the following
\begin{eqnarray}\label{volume-1}
\mu_0 (B(x_0, r)) \leq C{r^{P}}{\ln ^{Q}}r,
\end{eqnarray}
holds for all large enough $r$,
then problem (\ref{eq-pra})
admits no global positive solution either.
\end{corollary}

\begin{remark}\RM
Theorem \ref{thm:1} and Corollary \ref{cor:1} provide us an affirmative answer to the following question: how much could we relax the assumption on the volume growth of geodesic balls to ensure that problem
(\ref{eq-pra})
admits no global positive solution
 when  the nonlinear term $u^p$ is fixed?  In Section \ref{GE}, we show the sharpness of (\ref{volume-1}), which means that if we relax $P, Q$ a little,
there exists a global positive solution to (\ref{eq-pra}) on $M$ for small $u_0$.

Our method is to multiply the equation (\ref{eq-pra}) by $v^{a}\varphi^b $
( here $a, b$ are variable parameters).
 By building suitable integral estimates of $v$
 and choosing suitable test function $\varphi$, we can obtain the blow-up results.
 Actually, the test function $\varphi$ we use here can be considered as a parabolic version
 used in \cite{Sun-jmaa}.
%{\color{blue}{Corollary \ref{cor:1} is also obtained by \cite{MMP-P}, however, we use a quite different test function from theirs, and moreover, we also show the sharpness of $P, Q$, and in their paper, they did not.
%sharpness of $P, Q$.}}

\end{remark}

Corollary \ref{cor:1} can be presented in another equivalent form
\begin{corollary}\label{cor:2}\RM
For $V\equiv0$, if, for some reference point ${x_0}\in M$ and $\alpha>0$, the following
\begin{eqnarray}\label{volume-equiv}
\mu_0 (B(x_0, r)) \leq Cr^{\alpha}\ln ^{\frac{\alpha}{2}}r,
\end{eqnarray}
holds for all large enough $r$. If $1<p\leq 1+\frac{2}{\alpha}$, then  problem (\ref{eq-pra}) admits no global positive solution.
\end{corollary}

\begin{remark}\RM
Corollary \ref{cor:2} tells us if we know the upper bound of the volume of geodesic ball, then we can determine the range of $p$ to suffice that problem
 (\ref{eq-pra}) admits no global positive solution. Here the volume upper bound
 condition (\ref{volume-equiv})
 is also sharp, and can not be relaxed either, please see Theorems \ref{thm:3} and \ref{thm:4}.

 Corollary \ref{cor:2} is a generalization of Zhang's result, please see Theorem \ref{zhang-fujita}.
 Corollary \ref{cor:2} was first obtained by Mastrolia, Monticelli, and Punzo in \cite{MMP-P}.
\end{remark}
%We investigate the nonexistence of global positive solution of problem (\ref{eqT})
%via the upper bound of volume of geodesic balls, and  existence of global positive solution
%of problem (\ref{eqT}) by the upper bound of volume of geodesic balls
%and also the upper bound of heat kernel.

%The technique to multiply the equation (\ref{eq-pra}) by $v^{a}\varphi^b $, and to obtain an
% integral estimate involving $u$ is actually widely used, especially to
%deal with the elliptic equation and elliptic differential inequalities
%in the Euclidean space. This technique is called the nonlinear capacity method,
%which is systematically studied by Mitidieri and Pohozaev.
%Let us refer to \cite{Miti08, Miti09, Miti98, Miti01} for more details.

\bigskip

We then turn to study the existence of global solutions to problem (\ref{eq-pra}).
For that, we need  slightly
strengthen our assumptions on $M$. Let $\tilde{P}_t(x,y)$ be the smallest fundamental solution of the heat equation
$$\partial_t v=\tilde\Delta v\quad \mbox{on $M$}.$$
We know $\tilde{P}_t(x,y)$ is called the heat kernel of $\tilde{\Delta}$, and has the following properties
\begin{itemize}
\item
{Symmetry: $\tilde{P}_t(x,y)=\tilde{P}_t(y,x)$, for all $x,y\in M, t>0$.}

\item{Markovian property:
$\tilde{P}_t(x,y) \geq 0$, for all $x$, $y \in M$ and $t>0$, and
\begin{equation}\label{hk-markov}
\int_M \tilde{P}_t(x,y)d\mu(y) \leq 1,\quad\mbox{for all $x\in M$ and $t>0$}.
\end{equation}}

\item{ The semigroup identity: for all $x$, $y \in M$ and $t$, $s>0$,
\begin{equation}\label{hk-semigroup}
\tilde{P}_{t+s}(x,y) = \int_M \tilde{P}_t(x,z)\tilde{P}_s(z,y)d\mu(z).
\end{equation}}

\item{Approximation of identity: for any $f\in L^2(M,d\mu)$,
\begin{equation}\label{hk-iden}
\left\Vert\int_{M}\tilde{P_t}(x,y)f(y)d\mu(y)-f\right\Vert_{L^2(M,d\mu)}\to 0,\quad\mbox{as $t\to 0_+$}.
\end{equation} }
\end{itemize}
Let $P_t^{V}(x,y)$ denote the heat kernel of $-\Delta+V$ on $(M,\mu_0)$.
When $V=0$, we denote by $P_t(x,y):=P_t^{0}(x,y)$ the heat kernel of $\Delta$.
When $M$ has nonnegative Ricci curvature, by famous Li-Yau estimate in \cite{Li-Yau86}, we have
\begin{equation}
P_t(x,y)\asymp \frac{C}{\mu_0(x, \sqrt{t})}\exp{\left(-\frac{d^2(x,y)}{ct}\right).}
\end{equation}
Especially, when
$M=\mathbb{R}^N$
$$P_t(x,y)=\frac{1}{(4\pi t)^{\frac{N}{2}}}\exp\left(-\frac{|x-y|^2}{4t}\right).$$

The questions to obtain the lower bound and upper bound of heat kernels  $\tilde{P}_t(x,y)$ and $P_t^V(x,y)$ under different geometric
conditions on the underlying manifold have been extensively studied in the past few decades, let us
refer to the papers \cite{Chavel84, Davies89,Grigoryan06, Grigoryan-book,VSC92}.

We say $P_t^{V}$ satisfies the condition $(DUE)$, if $P_t^V$ has the following upper estimate
\begin{equation}%\label{heat kernel}
P_t^{V}(x,y)\leq \frac{C_1}{\mu_0(x, \sqrt{t})},\tag{$DUE$}
\end{equation}
for some constant $C_1$.

The heat kernels $P_t^V$ and $\tilde{P}_t$ are bridged by the following lemma.
\begin{lemma}\cite[Lemma 4.7]{Grigoryan06}\label{heat-rela}\RM\;
The heat kernels $P_t^V$ and $\tilde{P}_t$ have the following relation:
\begin{equation}
P_t^{V}(x,y)=\tilde{P}_t(x,y)h(x)h(y).
\end{equation}
\end{lemma}
If condition $(DUE)$ is satisfied on $M$, and $V$ is Green bounded and nonnegative, by Lemma \ref{heat-rela},
we have
\begin{equation}
\tilde{P}_t(x,y)\leq\frac{C}{\mu_0(x,\sqrt{t})}.
\end{equation}
for some constant $C$.

Our existence result is stated as follows.
\begin{theorem}\label{thm:3}\RM
Assume that $V\geq0$ is Green bounded and $P_t^V$ satisfies condition $(DUE)$.
If, for some $\varepsilon>0$, the following inequality
\begin{equation}\label{vol-e-1}
\mu_0(B(x_0,r))\geq cr^{P}\ln^{Q+\varepsilon}r,
\end{equation}
holds for all large enough $r$, then there exists a global positive solution to
(\ref{eq-pra}) for some small $u_0$. Here $P, Q$ are defined as in (\ref{def-ab}).
\end{theorem}

Theorem \ref{thm:3} also has an equivalent form.
\begin{theorem}\label{thm:4}\RM
Assume that $V\geq0$ is Green bounded and $P_t^V$ satisfies condition $(DUE)$.
Assume also, for some $\varepsilon>0$, the following inequality
\begin{equation}\label{vol-e-2}
\mu_0(B(x_0,r))\geq cr^{\alpha}\ln^{\frac{\alpha}{2}+\varepsilon}r,
\end{equation}
holds for all large enough $r$. If $p>1+\frac{2}{\alpha}$, then
there exists a global positive solution to
(\ref{eq-pra}) for some small $u_0$.
\end{theorem}

The paper is organized as follows: In Section \ref{example}, we present some examples to see the applications of our main
result.
 In Section \ref{sec-blow-up}, we give the proof of Theorem
\ref{thm:1}. In Section \ref{GE},  we present the proof of
Theorem \ref{thm:3}.

\begin{notation}\RM\;
 The letters $C,C^{\prime },C_{0},C_{1}, c_0, c_1...$ denote positive
constants whose values are unimportant and may vary at different occurrences.
\end{notation}

\section{Some examples}\label{example}

In this section we present several examples to show the applications of Theorem \ref{thm:1} and
Corollary \ref{cor:2}.

First, let us make some preliminary works.
Define the Riesz potential on $\mathbb{R}^N$ for $0<\alpha<N$ by
\begin{equation}\label{RZ}
I_{\alpha}f(x)=c(N, \alpha)\int_{\mathbb{R}^N}\frac{f(y)}{|x-y|^{N-\alpha}}dy,
\end{equation}
where $f\in L_{loc}^1(\mathbb{R}^N)$, and $\int_{|x|\geq1}|x|^{N-\alpha}|f(x)|dx<\infty$, and
$$c(N, \alpha)=\frac{\Gamma(\frac{N-\alpha}{2})}{\pi^{\frac{N}{2}}2^{\alpha}\Gamma(\frac{\alpha}{2})}.$$
Here $\Gamma(\cdot)$ is the Gamma function.

\begin{lemma}
\label{HMV}
\cite[Corollary 2.9]{HMV99}\RM\;
If $V\leq 0$, and there exists some constant $C_2(N)$ such that
\begin{equation}\label{ineq-iter}
I_1[(I_1V)^2](x)\leq -C_2(N)I_1V(x),
\end{equation}
then there exists a positive solution $h$ to
$$\Delta h=Vh\quad\mbox{in $\mathbb{R}^N$}.$$
 Moreover, if
$I_2(-V)<\infty$, then the solution $h$ satisfies
\begin{equation}\label{h-est}
\exp(-I_2V)\leq h\leq \exp(-C_3I_2V),
\end{equation}
for some constant $C_3=C_3(N)>0$.
\end{lemma}

\begin{proposition}\label{Zhang-bdd}
\cite[Proposition 2.1]{Zhang-98}\RM\;
Let $V(x)=\frac{1}{1+|x|^b}$ for some $b>2$. Then
\begin{equation}\label{sup-bdd}
\sup_{x\in\mathbb{R}^N}I_2V(x)<\infty.
\end{equation}
\end{proposition}

\begin{proposition}\label{Prop-b>2}\RM\;
Let $V(x)=\frac{1}{1+|x|^b}$ for some $b>2$. Then there exists a constant $C(N,b)>0$ such that for all $x\in \mathbb R^N$,
\begin{equation}\label{eq1}
I_1 [(I_1V)^2](x)\leq C(N,b)I_1 V(x).
\end{equation}
\end{proposition}
\begin{proof}
We divide the proof into two steps.

\textbf{Step 1.\;}
We show the following estimate
\begin{equation}\label{eq2}
I_1V(x)\asymp\left\{
              \begin{array}{ll}
                1, & \hbox{$|x|\leq1$;} \\
                 |x|^{1-b}+|x|^{1-N}\left(1+\int_{1}^{|x|}r^{N-b-1}dr\right), & \hbox{$|x|>1$.}
              \end{array}
            \right.
\end{equation}
By definition of $I_1V$, we have
\begin{equation}\label{eq3}
I_1V(x)=C(N)\int_{\mathbb R^N}\frac{dy}{(1+|x-y|^b)|y|^{N-1}}.
\end{equation}

\medskip

Firstly, we deal with the case that $|x|\leq1$. The integral of the right hand side of \eqref{eq3} can be written as
\begin{eqnarray}\label{eq4}
\int_{\mathbb R^N}\frac{dy}{(1+|x-y|^b)|y|^{N-1}}&=&\int_{|y|\leq2|x|}\frac{dy}{(1+|x-y|^b)|y|^{N-1}}
+\int_{|y|>2|x|}\frac{dy}{(1+|x-y|^b)|y|^{N-1}}\nonumber\\
&=:&J_1+J_2.
\end{eqnarray}
Then for $|y|\leq2|x|$, we have $|y-x|\leq3|x|\leq 3$, and $1+|x-y|^b\asymp1$. Using polar coordinates,
we obtain
\begin{equation*}
J_1\asymp\int_{|y|\leq2|x|}\frac{dy}{|y|^{N-1}}\asymp\int_0^{2|x|}\frac{r^{N-1}}{r^{N-1}}dr\asymp|x|.
\end{equation*}
For $|y|>2|x|$, we have $|y|/2\leq |y-x|\leq 3|y|/2$, and $1+|x-y|^b\asymp 1+|y|^b$, then by the fact $b>2$,
we obtain
\begin{equation*}
J_2\asymp\int_{|y|\geq2|x|}\frac{dy}{(1+|y|^b)|y|^{N-1}}=C(N)\int_{2|x|}^{\infty}\frac{dr}{1+r^{b}}\asymp1.
\end{equation*}
By substituting the two estimates to \eqref{eq4}, we obtain
\begin{equation*}
I_1V(x)\asymp1+|x|\asymp1,
\end{equation*}
which is the first estimate in \eqref{eq2}.

Secondly, when $|x|>1$, let us write the integral in \eqref{eq3} as
\begin{equation}\label{eq5}
\int_{\mathbb R^N}\frac{dy}{(1+|x-y|^b)|y|^{N-1}}=\int_{|y|\leq|x|/2}+\int_{|x|/2<|y|\leq2|x|}+\int_{|y|>2|x|}
=:K_1+K_2+K_3.
\end{equation}
Then we estimate $K_1,K_2,K_3$ respectively.

 For $|y|\leq|x|/2$, we have $|y-x|\asymp |x|$. Thus
\begin{equation*}
K_1\asymp\frac1{|x|^b}\int_{|y|\leq |x|/2}\frac{dy}{|y|^{N-1}}\asymp |x|^{1-b}.
\end{equation*}
For $|x|/2<|y|\leq 2|x|$, we have $|y|\asymp |x|$. Thus
\begin{equation*}
K_2\asymp \frac1{|x|^{N-1}}\int_{|x|/2<|y|\leq 2|x|}\frac{dy}{1+|x-y|^b}.
\end{equation*}
Noting that $\{y:\ |y-x|\leq|x|/2\}\subseteq\{y:\ |x|/2<|y|\leq 2|x|\}\subseteq\{y:\ |y-x|\leq3|x|\}$, we have
\begin{align*}
&\int_{|x|/2<|y|\leq 2|x|}\frac{dy}{1+|x-y|^b}\leq\int_{|z|\leq 3|x|}\frac{dz}{1+|z|^b}\asymp \int_0^3r^{N-1}dr+\int_3^{3|x|}r^{N-b-1}dr\\
&\asymp 1+\int_1^{|x|}r^{N-b-1}dr,
\end{align*}
and similarly,
\begin{equation*}
\int_{|x|/2<|y|\leq 2|x|}\frac{dy}{1+|x-y|^b}\geq\int_{|z|\leq |x|/2}\frac{dz}{1+|z|^b}\asymp 1+\int_{1}^{|x|}r^{N-b-1}dr.
\end{equation*}
Combining the above estimates, we obtain
\begin{equation*}
K_2\asymp |x|^{1-N}\left(1+\int_1^{|x|}r^{N+b-1}dr\right).
\end{equation*}
For $|y|>2|x|$, we have $|y-x|\asymp|y|$. Thus
\begin{equation*}
K_3\asymp \int_{|y|> 2|x|}\frac{dy}{(1+|y|^b)|y|^{N-1}}\asymp\int_{|y|> 2|x|}\frac{dy}{|y|^{b+N-1}}\asymp  |x|^{1-b}.
\end{equation*}

By substituting the estimates of $K_1,K_2,K_3$ into \eqref{eq5}, we obtain the second estimate in \eqref{eq2}.

\textbf{Step 2.\;}
Now we apply \eqref{eq2} to show \eqref{eq1} with $V=\frac1{1+|x|^b}$. We separate the proof into two cases.

\textbf{Case of $N\geq b$.\;}We show that there is a constant $C=C(N,b)>0$ such that for all $x\in \mathbb R^N$,
\begin{equation}\label{eq6}
(I_1V(x))^2\leq C V(x),
\end{equation}
and \eqref{eq1} follows immediately by taking $I_1$ on both sides of \eqref{eq6}.

When $|x|\leq1$,  \eqref{eq6} is true, since we have
\begin{equation*}
(I_1V)(x)\asymp V(x)\asymp 1.
\end{equation*}
When $|x|>1$,  by \eqref{eq2}, we have
\begin{equation*}
(I_1V)(x)\asymp\left\{
                 \begin{array}{ll}
                   (1+\log|x|)|x|^{1-b}, & \hbox{$N=b$,} \\
                   |x|^{1-b}, & \hbox{$N>b$.}
                 \end{array}
               \right.
\end{equation*}
Noting $b>2$, we have
\begin{equation*}
(I_1V(x))^2\leq C(1+\log|x|)^2|x|^{2-2b}\leq C|x|^{-b}\leq CV(x),
\end{equation*}
which proves \eqref{eq6}.

\textbf{Case of $N< b$.\;} By \eqref{eq2}, we have
\begin{equation}\label{eq7}
I_1V(x)\asymp\left\{
              \begin{array}{ll}
                1, & \hbox{$|x|\leq1$,} \\
                 |x|^{1-N}, & \hbox{$|x|>1$.}
              \end{array}
            \right.
\end{equation}
By definition
\begin{equation}\label{eq8}
I_1[(I_1V)^2](x)= C(N)\int_{\mathbb R^N}\frac{(I_1V)^2(y)dy}{|x-y|^{N-1}}.
\end{equation}

Let us first consider $|x|\leq1$. Applying \eqref{eq7}, we obtain
\begin{eqnarray*}
\int_{\mathbb R^N}\frac{(I_1V)^2(y)dy}{|x-y|^{N-1}}
&=&\int_{|y|\leq2}\frac{(I_1V)^2(y)dy}{|x-y|^{N-1}}+\int_{|y|>2}\frac{(I_1V)^2(y)dy}{|x-y|^{N-1}}\\
&\asymp& \int_{|y|\leq2}\frac{dy}{|x-y|^{N-1}}+\int_{|y|>2}|y|^{2-2N}\frac{dy}{|x-y|^{N-1}}\\
&\leq&\int_{|z|\leq3}\frac{dz}{|z|^{N-1}}+\int_{|y|>2}\frac{dy}{|y|^{3N-3}}\\
&\asymp& 1,
\end{eqnarray*}
which together with $I_1V(x)\asymp 1$, implies for $|x|\leq1$,
\begin{equation*}
I_1[(I_1V)^2](x)\leq CI_1V(x).
\end{equation*}
Then we consider $|x|>1$. Rewrite the integral in \eqref{eq8} as
\begin{eqnarray*}
\int_{\mathbb R^N}\frac{(I_1V)^2(y)dy}{|x-y|^{N-1}}&=&\left(\int_{|y|\leq1/2}+\int_{1/2<|y|\leq|x|/2}
+\int_{|x|/2<|y|\leq2|x|}+\int_{|y|>2|x|}\right)\frac{(I_1V)^2(y)dy}{|x-y|^{N-1}}\\
&=:&L_1+L_2+L_3+L_4.
\end{eqnarray*}
We estimate $L_i(i=1,2,3,4)$ as follows.

For $|y|\leq1/2$, we have by \eqref{eq7} that $I_1V(y)\asymp 1$, thus
\begin{equation*}
L_1\asymp \int_{|y|\leq1/2}\frac{dy}{|x-y|^{N-1}}\asymp\int_{|y|\leq1/2}\frac{dy}{|x|^{N-1}}\asymp |x|^{1-N}.
\end{equation*}

For $1/2<|y|\leq|x|/2$, we have by \eqref{eq7} that $(I_1V(y))^2\asymp |y|^{2-2N}$, and $|x-y|\asymp |x|$, thus
\begin{equation*}
L_2\asymp \frac1{|x|^{N-1}}\int_{1/2<|y|\leq|x|/2}{|y|^{2-2N}dy}\asymp \frac{1-|x|^{2-N}}{|x|^{N-1}}.
\end{equation*}

For $|x|/2<|y|\leq2|x|$, we have $(I_1V(y))^2\asymp |y|^{2-2N}\asymp |x|^{2-2N}$, thus
\begin{equation*}
L_3\asymp {|x|^{2-2N}}\int_{|x|/2<|y|\leq2|x|}\frac{dy}{|x-y|^{N-1}}\asymp |x|^{3-2N}.
\end{equation*}

For $|y|>2|x|$, we have $(I_1V(y))^2\asymp |y|^{2-2N}$, and $|x-y|\asymp |y|$, thus
\begin{equation*}
L_4\asymp \int_{|y|>2|x|}\frac{|y|^{2-2N}dy}{|y|^{N-1}}\asymp {|x|^{3-2N}}.
\end{equation*}

Combing the above estimates, we obtain
\begin{align*}
I_1[(I_1V)^2](x)&\asymp L_1+L_2+L_3+L_4\\
&\asymp |x|^{1-N}+\frac{1-|x|^{2-N}}{|x|^{N-1}}+|x|^{3-2N}+{|x|^{3-2N}}\\
&\asymp |x|^{1-N}.
\end{align*}
Thus applying \eqref{eq7}, we obtain for $|x|>1$,
\begin{equation*}
I_1[(I_1V)^2](x)\leq CI_1V(x).
\end{equation*}
Hence, \eqref{eq1} also holds for the case of $N<b$. The proof is complete.
\end{proof}

\begin{lemma}\label{h-negative}\RM\;
If $V(x)=\frac{\omega}{1+|x|^b}$ for some $\omega<0$ and $b>2$, then there exists a
positive solution $h$ to
$$\Delta h=V(x)h.$$
Moreover, $h\asymp1$.
\end{lemma}
\begin{proof}
Combining Lemma \ref{HMV} and Proposition \ref{Prop-b>2}, we obtain there exists a positive solution
$h$ to
$$\Delta h=Vh,$$
and by Proposition \ref{Zhang-bdd}, we have
\begin{equation}
\sup_{x\in \mathbb{R}^N}I_2(-V)<\infty.
\end{equation}
Hence, from (\ref{h-est}), we obtain
\begin{equation*}
h\asymp1.
\end{equation*}
\end{proof}

\begin{lemma}\cite[Lemma 2.2]{IK08}
\label{h-est-IK}\RM\;
Assume that $V$ satisfies the following conditions for some $\omega>0$ and $\theta>0$
\begin{enumerate}
\item{$V=V(|x|)\in C^1(\mathbb{R}^N)$, and $V(r)\geq0$ on $[0, \infty)$,}
\item{$\sup\limits_{r\geq1}r^{2+\theta}|V(r)-\frac{\omega}{r^2}|<\infty$,}
\item{$\sup\limits_{r\geq1}|r^3V^{\prime}(r)|<\infty$.}
\end{enumerate}
Then there exists a unique $C^2$ solution $h(r)>0$ to
\begin{equation*}
\Delta h=hV,\quad\mbox{in $\mathbb{R}^N$},
\end{equation*}
such that
\begin{equation}\label{h-est}
h(r)\asymp r^{\alpha(\omega)},\quad \mbox{for large enough $r$}.
\end{equation}
where
$\alpha(\omega)$ is defined as in (\ref{root-1}).
%\begin{equation}
%\alpha(\omega)=\frac{-(N-2)+\sqrt{(N-2)^2+4\omega}}{2},
%\end{equation}
\end{lemma}

\begin{example}\RM
Let $V(x)=0$, and  $M=\mathbb{R}_{g}^k\times \mathbb{S}^l$ be endowed with product metric.
 Here $\mathbb{R}_g^k=(\mathbb{R}^k, g)$ is a model manifold with induced metric
$g=dr^2+\psi(r)^2d\theta^2$, where $(r,\theta)$ is the polar coordinates in $\mathbb{R}^k$, and
$\psi(r)$ is a smooth, positive  function on $(0,\infty)$ such that
\begin{equation*}
\psi(r)=\left\{
\begin{array}{ll}
r, \quad\mbox{for small $r$}, \\
\left(r^{\alpha-1}\ln^{\frac{\alpha}{2}}r\right)^{\frac{1}{k-1}},\quad\mbox{for large $r$}.
\end{array}
\right.
\end{equation*}
If $V(x)=0$, we could choose $h=1$, and hence, in (\ref{nu-meas}) $d\nu=d\mu=d\mu_0$.
Then the volume of the ball $B_r:=B_r(0)$ in $\mathbb{R}_g^k$ can be determined by
\begin{eqnarray*}
\mu_0(B_r)=\int_0^r S(\tau)d\tau,
\end{eqnarray*}
where
$S$ is the surface area defined by
\begin{equation*}
S(r)=\left\{
\begin{array}{ll}
r^{k-1}, \quad\mbox{for small $r$}, \\
r^{\alpha-1}\ln^{\frac{\alpha}{2}} r,\quad\mbox{for large $r$}.
\end{array}
\right.
\end{equation*}
Hence, we obtain
\begin{eqnarray*}
\mu_0(B_r)\leq Cr^{\alpha}\ln^{\frac{\alpha}{2}}r,\quad\mbox{for large enough $r$}.
\end{eqnarray*}
If follows that the geodesic ball $B(0, r)$ in M satisfies
\begin{eqnarray*}
\mu_0(B(0, r))\leq Cr^{\alpha}\ln^{\frac{\alpha}{2}}r, \quad\mbox{for large enough $r$}.
\end{eqnarray*}
Applying Corollary \ref{cor:2}, we derive that when $p\leq 1+\frac{2}{\alpha}$, then
 (\ref{eq-pra}) on $\mathbb{R}_g^k\times \mathbb{S}^l$ admits no global positive solution. Especially, when $\mathbb{R}_{g}^k=\mathbb{R}^k$, we know that the critical exponent
for $\mathbb{R}^k\times\mathbb{S}^l$ is
$1+\frac{2}{k}$.
\end{example}

\begin{example}\RM
When $M=\mathbb{R}^N$, we consider the following classes of $V(x)$.
\begin{enumerate}
\item{If $0\leq V(x)\leq \frac{\omega}{1+|x|^b}$ for some $b>2$ and $\omega>0$, we know
by Proposition \ref{Zhang-bdd}
\begin{equation*}
\sup_{x\in \mathbb{R}^N}I_2V<\infty,
\end{equation*}
which means that $V(x)$ is Green bounded. By Lemma \ref{h-exist}, we know that
\begin{equation*}
\Delta h=Vh,
\end{equation*}
admits a solution
$h\asymp1$.
Noting that
\begin{equation}
\nu(B(0,r))=\int_{B(0, r)}hdx\asymp r^{N}.
\end{equation}
By Theorem \ref{thm:1}, we know if
$$\nu(B(0, r))\leq Cr^{P}\ln^Qr,$$
or more precisely, when
$$p\leq 1+\frac{2}{N}.$$
there exists no global solution to (\ref{eq-pra}).

This result also covers the result (1) of Theorem \ref{thm-Zhang}.
}

\item[(2)]{When $0\leq V(x)\leq\frac{\omega}{|x|^2}$, for large $|x|$, by employing Comparison principle, we can replace $V(x)$ by $\frac{\omega}{|x|^2}(1+|x|^{-\theta})$ for large $|x|$ still denoted by $V(x)$.
    If we can show that (\ref{eq-pra}) admits no global positive solution with
    $V(x)=\frac{\omega}{|x|^2}(1+|x|^{-\theta})$ for large $|x|$, then the original problems admits
    no global positive solution by Comparison principle.

    Applying Lemma \ref{h-est-IK}, we know the following problem with $V(x)=\frac{\omega}{|x|^2}(1+|x|^{-\theta})$ for large $|x|$
\begin{equation}
\Delta h=Vh, \quad\mbox{in $\mathbb{R}^N$},
\end{equation}
admits a unique solution $h>0$ such that
\begin{equation}\label{h-app-1}
h(x)\asymp|x|^{\alpha(\omega)},\mbox{for large $|x|$}.
\end{equation}
where $\alpha(\omega)$ is defined as in (\ref{root-1}).

For large $r$, we obtain that
\begin{equation}\label{vol-h}
\nu(B(0, r))=\int_{B(0, r)}hdx\asymp r^{N+\alpha(\omega)},\quad\mbox{for large $r$}.
\end{equation}
Applying Theorem \ref{thm:1}, we obtain that
when $p\leq1+\frac{2}{N+\alpha(\omega)}$, there exists no global positive solution to (\ref{eq-pra}).

In this case, the result is also in accordance with the (2) in Theorem \ref{thm-Ishige}.
}

\item[(3)]{When $\frac{\omega}{1+|x|^b}\leq V(x)\leq0$ for some $\omega<0$, and $b>2$.
By Lemma \ref{h-negative}, we obtain that $h\asymp1$. Applying Theorem \ref{thm:1}, we obtain that
when $p\leq 1+\frac{2}{N}$, there exists no global positive solution to (\ref{eq-pra}).

In this case, the result is also in accordance with the (3) in Theorem \ref{thm-Zhang}. Actually, we remove the restriction that $\omega$ is small enough, and we improve the result obtained in Theorem \ref{thm-Zhang}.}

\item[(4)]{ When $V=
    \frac{\alpha(\omega)N+\omega|x|^2}{(1+|x|^2)^2}$ for $\omega\in[-\frac{(N-2)^2}{4},0)$, we know
$$V(x)\leq \frac{\omega}{1+|x|^2},$$
and $\Delta h=Vh$ admits a solution
 $h(x)=(1+|x|^2)^{\frac{\alpha(\omega)}{2}}$.
    By Theorem \ref{thm:1}, we know if $p\leq1+\frac{2}{N+\alpha(\omega)}$, there exists no global positive solution to (\ref{eq-pra}).
    }
\end{enumerate}
\end{example}
\begin{remark}\RM
Here we can not cover the case of $V(x)\leq\frac{\omega}{1+|x|^2}\leq0$, the difficulty is that we
do not know the
asymptotic behavior of $h$ when $|x|\to\infty$. However, we conjecture that when $V(x)$ behaves like $\frac{\omega}{1+|x|^2}$, $\Delta h=Vh$ admits a solution $h\asymp|x|^{\alpha(\omega)}$
for large $|x|$.
\end{remark}

\begin{example}\RM
Assume that $M$ satisfies
\begin{equation}
\mu_0(B(x_0, r))\leq Cr^{\alpha},\quad\mbox{for large enough $r$},
\end{equation}
and
\begin{equation*}
G(x,y)\asymp d(x,y)^{2-\alpha},\quad\mbox{for large enough $d(x,y)$}.
\end{equation*}
Let $V(x)\asymp \frac{\omega}{1+|x|^b}$ for $b>2$, and $\omega>0$, we know
\begin{equation}
\sup_{x}\int_{M}G(x,y)V(y)dy<\infty,
\end{equation}
hence $V(x)$ is Green bounded. Hence by Lemma \ref{h-exist}, we know there exists a function
$h(x)\asymp1$ satisfying $\Delta h=Vh$. Applying Theorem \ref{thm:1}, we know if
\begin{equation}
p\leq 1+\frac{2}{\alpha}
\end{equation}
then there is no positive global solution to (\ref{eq-pra}).
\end{example}

\section{Nonexistence of global positive solution}\label{sec-blow-up}

\bigskip

\begin{proof}[Proof of Theorem \ref{thm:1}]\RM
Let $\varphi \in W_{c}^{1, 2}(M\times[0, +\infty ), d\mu dt)$ be a function satisfying $0\le \varphi \le 1$, $\varphi \equiv 1$ in a neighborhood of $D_R:=\overline{B(x_0, R)}\times[0, {R^2}]$. Define
\begin{eqnarray}\label{def-test}
\psi (x, t) = v(x,t)^{-a}\varphi(x, t){^b},
\end{eqnarray}
where $a$ will take arbitrarily small positive value near zero, and $b$ will be chosen to be a large enough fixed constant.

Without loss of generality, let us assume that $1/v$ is locally bounded, otherwise we can replace $v$ by $v+\varepsilon$ for $\varepsilon>0$, at last we can let $\varepsilon\to0_{+}$. From (\ref{def-test}), we know
that $\psi$ has compact support and is bounded.
Note that
\begin{eqnarray}\label{eq:grad_psi}
\nabla \psi = b v^{ - a }\varphi ^{b - 1}\nabla \varphi - a v^{ - a - 1}\varphi ^b \nabla v,
\end{eqnarray}
and
\begin{eqnarray}\label{eq:partial_psi}
{\partial _t}\psi = b v^{ - a }\varphi ^{b - 1}\partial _t\varphi - a v^{ - a - 1}\varphi ^b \partial_tv.
\end{eqnarray}
Thus
$$\psi \in W_{c}^{1, 2}(M\times[0, +\infty ), d\mu dt).$$
Substituting (\ref{eq:grad_psi}) into (\ref{def: weak-sol-ineq}),  we obtain
\begin{eqnarray}\label{eq:the_second_inequality}
\int_0^\infty \int_M h^{p-1}v^p\psi d\mu dt &+& a\int_0^\infty \int_M v^{-a - 1}\left| {\nabla v} \right|^2\varphi ^b d\mu dt\nonumber\\
&\leq& \int_0^\infty \int_M (\nabla v, b v^{ - a}\varphi ^{b - 1}\nabla \varphi )d\mu dt - \int_0^\infty \int_M v\partial _t\psi d\mu dt .
\end{eqnarray}
Applying the Young's inequality to the first term in the right-hand side of (\ref{eq:the_second_inequality}), we obtain
\begin{eqnarray*}
&&\int_0^\infty \int_M (\nabla v, b v^{-a}\varphi^{b - 1}\nabla \varphi)d\mu dt  \nonumber\\
&=& \int_0^\infty \int_M (a^{\frac{1}{2}}v^{\frac{- a-1}{2}}\varphi^{\frac{b}{2}}\nabla v, ba^{-\frac{1}{2}}
v^{\frac{1-a}{2}}\varphi^{\frac{b}{2}-1}\nabla\varphi )d\mu dt \nonumber\\
&\leq& \frac{a}{2}\int_0^{\infty}\int_M v^{-a-1}\left|\nabla v \right|^2\varphi^{b}d\mu dt+
\frac{b^2}{2a}\int_0^\infty\int_{M} v^{1-a}\varphi^{b-2}\left|\nabla \varphi\right|^2d\mu dt.
\end{eqnarray*}
Substituting the above into (\ref{eq:the_second_inequality}), we obtain
\begin{eqnarray}\label{right-12}
&& \int_0^\infty \int_M h^{p-1}v^p\psi d\mu dt + \frac{a}{2}\int_0^\infty \int_M v^{-a-1}\left|\nabla v\right|^2\varphi^{b}d\mu dt \nonumber \\
&\leq&\frac{b^2}{2a}\int_0^\infty \int_M v^{1-a}\varphi^{b - 2}\left|\nabla\varphi \right|^2d\mu dt
- \int_0^\infty \int_M v\partial_t\psi d\mu dt.
\end{eqnarray}
Combining (\ref{right-12}) with (\ref{eq:partial_psi}), we obtain
\begin{eqnarray*}
 \int_0^\infty \int_M v^{-a-1}\left|\nabla v \right|^2\varphi^b d\mu dt
&\leq& \frac{b^2}{a^2}\int_0^\infty\int_{M} v^{1-a}\varphi^{b - 2}\left|\nabla \varphi \right|^2d\mu dt\nonumber\\
&&- \frac{2}{a}\int_0^\infty \int_M v[b v^{-a }\varphi^{b - 1}\partial_t\varphi - a v^{-a-1}\varphi^b\partial_tv]d\mu dt,
\end{eqnarray*}
which is
\begin{eqnarray}\label{eq:estimate_of_grad_u}
\int_0^\infty \int_M v^{-a-1}\left|\nabla v\right|^2\varphi^{b}d\mu dt
&\leq&\frac{b^2}{a^2}\int_0^\infty \int_M v^{1-a}\varphi^{b-2}\left|\nabla \varphi\right|^2d\mu dt \nonumber \\
&&- \frac{2}{a}\int_0^{\infty}\int_M (b v^{1-a}\varphi^{b-1}\partial_t\varphi - a v^{-a}\varphi^{b}\partial_tv)d\mu dt.
\end{eqnarray}
Let us use another feasible test function $\psi(x, t) =\varphi(x, t)^{b}$.
Substituting $\psi =\varphi^{b}$ into (\ref{def: weak-sol-ineq}), we obtain
\begin{eqnarray}\label{eq:the_third_inequality}
\int_0^\infty \int_M h^{p-1}v^p\varphi^{b} d\mu dt\leq \int_0^{\infty} \int_M (\nabla v, b\varphi^{b - 1}\nabla \varphi)d\mu dt - b\int_0^\infty\int_M  v\varphi^{b - 1}\partial_t\varphi d\mu dt.
\end{eqnarray}
Let us estimate the first term in the right-hand side of (\ref{eq:the_third_inequality}) via  the Young's inequality
\begin{eqnarray*}
&&\int_0^\infty \int_M (\nabla v, b\varphi^{b-1}\nabla\varphi)d\mu dt\nonumber\\
&=&\int_0^\infty\int_M (a^{\frac{1}{2}}v^{\frac{-a- 1}{2}}\varphi^{\frac{b}{2}}\nabla v, ba^{-\frac{1}{2}}
v^{\frac{a+ 1}{2}}\varphi^{\frac{b}{2}-1}\nabla \varphi)d\mu dt \nonumber\\
&\leq& \frac{a}{2}\int_0^{\infty} \int_Mv^{-a-1}\left|\nabla v\right|^2\varphi^{b}d\mu dt+
\frac{b^2}{2a}\int_0^{\infty}\int_M v^{a+1}\varphi^{b-2}\left|\nabla\varphi\right|^2d\mu dt.
\end{eqnarray*}
Combining the above with (\ref{eq:the_third_inequality}), we obtain
\begin{eqnarray*}
\int_0^{\infty}\int_{M} h^{p-1}v^p\varphi^{b}d\mu dt &\leq&
\frac{a}{2}\int_0^{\infty}\int_{M}v^{-a-1}\left|\nabla v\right|^2
\varphi^{b}d\mu dt \nonumber\\
&& + \frac{b^2}{2a}\int_0^{\infty}\int_{M}v^{a+1}\varphi^{b-2}\left|\nabla\varphi\right|^2d\mu dt
-b\int_0^{\infty}\int_{M} v\varphi^{b-1}\partial_t\varphi d\mu dt.
\end{eqnarray*}
Substituting (\ref{eq:estimate_of_grad_u}) into the above, we obtain
\begin{eqnarray}\label{eq:estimate_of_I_by_four_K}
\int_0^{\infty}\int_M h^{p-1}v^p\varphi^{b}d\mu dt &\leq& \frac{b^2}{2a}\int_0^\infty
\int_M v^{1-a}\varphi^{b-2}\left|\nabla\varphi\right|^2d\mu dt\nonumber \\
&& - \int_0^{\infty}\int_M (b v^{1-a}\varphi^{b-1}\partial_t\varphi-a v^{-a}\varphi^{b}\partial_tv )d\mu dt \nonumber\\
&& + \frac{b^2}{2a}\int_0^\infty\int_M v^{a + 1}\varphi^{b-2}\left|\nabla\varphi \right|^2d\mu dt \nonumber \\
&&-b\int_0^{\infty}\int_M v\varphi^{b-1}\partial_t\varphi d\mu dt.
\end{eqnarray}
For convenience, let us denote
\begin{eqnarray*}
I&:=&\int_0^{\infty}\int_M h^{p-1}v^p\varphi^{b}d\mu dt,\\
K_1&:=&\frac{b^2}{2a}\int_0^\infty
\int_M v^{1-a}\varphi^{b-2}\left|\nabla\varphi\right|^2d\mu dt,\\
K_2&:=& - \int_0^{\infty}\int_M (b v^{1-a}\varphi^{b-1}\partial_t\varphi-a v^{-a}\varphi^{b}\partial_tv )d\mu dt,\\
K_3&:=& \frac{b^2}{2a}\int_0^\infty\int_M v^{a + 1}\varphi^{b-2}\left|\nabla\varphi \right|^2d\mu dt,\\
K_4&:=&-b\int_0^{\infty}\int_M v\varphi^{b-1}\partial_t\varphi d\mu dt.
\end{eqnarray*}
Then (\ref{eq:estimate_of_I_by_four_K}) can be written as follows
\begin{eqnarray}\label{IK}
I\leq K_1+K_2+K_3+K_4.
\end{eqnarray}
Before estimating (\ref{IK}), let us introduce some notations
\begin{eqnarray}\label{eq:definition_IJQ}
J(\theta_1,\theta_2 ) := \int_{0}^{\infty }\int_{M}
h^{\theta_1}\left|\nabla\varphi\right|^{\theta_2}d\nu dt,
 \quad
L(\theta_1,\theta_2 ) := \int_{0}^{\infty }\int_{M}h^{\theta_1}
\left|\partial_t \varphi\right|^{\theta_2}d\nu dt.
\end{eqnarray}
Noting $\varphi \equiv 1$ in a neighborhood of $D_R$, and applying the H\"{o}lder's inequality,
we obtain
\begin{eqnarray*}
K_1&=& \frac{b^2}{2a}\iint_{D_R^c} v^{1-a}\varphi^{b-2}\left|\nabla\varphi\right|^2
d\mu dt \nonumber\\
&=& \frac{b^2}{2a}\iint_{D_R^c}
(h^{\frac{(p-1)(1-a)}{p}}v^{1-a}\varphi^{\frac{b(1-a)}{p}})(h^{-\frac{(p-1)(1-a)}{p}}
\varphi^{-\frac{b(1-a)}{p} +b-2}\left|\nabla\varphi\right|^2)d\mu dt \nonumber\\
&\leq& \frac{b^2}{2a}\left(\iint_{D_R^c}h^{p-1}v^p\varphi^{b}d\mu dt\right)^{\frac{1-a}{p}}\nonumber\\
&&\times\left(\int_0^{\infty} \int_M h^{-\frac{(p-1)(1-a)}{p+a-1}}\varphi^{[-\frac{b(1-a)}{p}+b-2]\frac{p}{p+a-1}}\left|\nabla\varphi\right|^{\frac{2p}{p+a-1}}
d\mu dt\right)^{\frac{p+a
-1}{p}}\nonumber\\
&\leq&\frac{b^2}{2a}\left(\iint_{D_R^c}h^{p-1}v^p\varphi^{b}d\mu dt\right)^{\frac{1-a}{p}}
\nonumber\\
&&\times\left(\int_0^{\infty} \int_M h^{\frac{ap}{p+a-1}}\varphi^{[-\frac{b(1-a)}{p}+b-2]\frac{p}{p+a-1}}\left|\nabla\varphi\right|^{\frac{2p}{p+a-1}}
d\nu dt\right)^{\frac{p+a
-1}{p}}.
\end{eqnarray*}
Here we have used that $D_R^c=M\times[0,\infty)\setminus D_R$, and
$d\nu=h^{-1}d\mu$.

Noting that $0\leq\varphi \leq 1$, and by choosing sufficiently large $b$, we obtain
\begin{eqnarray}\label{eq:estimate_K1}
{K_1} \leq \frac{C}{a}\left(\iint_{D_R^c}h^{p-1}v^p\varphi^{b}d\mu dt\right)^{\frac{1-a}{p}}
J\left(\frac{ap}{p+a-1},
\frac{2p}{p+a-1}\right)^{\frac{p+a-1}{p}}.
\end{eqnarray}
Applying integration by parts to $K_2$, we obtain
\begin{eqnarray*}
K_2&=&-\int_0^{\infty}\int_M(b v^{1-a}\varphi^{b-1}\partial_t\varphi-a v^{-a}\varphi^{b}\partial_tv )d\mu dt\nonumber\\
&=&-\int_0^{\infty}\int_M v^{1-a}\partial_t (\varphi^{b})-\frac{a}{1-a}\partial_t(v^{1-a})\varphi^{b} d\mu dt\nonumber\\
&=& - \int_0^\infty \int_M v^{1-a}\partial_t (\varphi^b)d\mu dt -
\frac{a}{1- a}\int_M v_0^{1-a}\varphi(x,0)^b d\mu\nonumber\\
&& - \frac{a}{1-a}\int_0^{\infty}
\int_M v^{1-a}\partial_t (\varphi^b)d\mu dt \nonumber\\
&\leq& -\frac{1}{1-a}\int_0^\infty \int_M v^{1-a}\partial_t(\varphi^{b})d\mu dt \nonumber\\
&=& -\frac{b}{1-a}\iint_{D_R^c} v^{1-a}\varphi^{b - 1}\partial_t\varphi d\mu dt.
\end{eqnarray*}
Using H\"{o}lder's inequality again and by similar arguments as in $K_1$, we obtain
\begin{eqnarray}\label{eq:estimate_K2}
K_2&\leq & \frac{b}{1- a}\iint_{D_R^c} (h^{\frac{(p-1)(1-a)}{p}}v^{1-a}\varphi^{\frac{1-a}{p}b})
(h^{-\frac{(p-1)(1-a)}{p}}\varphi^{b - 1-\frac{1- a}{p}b}\left|\partial_t\varphi \right|)d\mu dt\nonumber\\
&\leq& \frac{b}{1-a}\left(\iint_{D_R^c}h^{p-1} v^p\varphi^{b}d\mu dt\right)^{\frac{1-a}{p}}\nonumber\\
&&\times\left(\int_0^{\infty} \int_M
h^{-\frac{(p-1)(1-a)}{p+a-1}}\varphi^{[ b-1- \frac{1-a}{p}b]
\frac{p}{p+a-1}}\left|\partial_t\varphi\right|^{\frac{p}{p+a-1}}d\mu dt\right)^{\frac{p+a-1}{p}} \nonumber\\
&\leq&\frac{C}{1-a}\left(\iint_{D_R^c} h^{p-1}v^p\varphi^{b}d\mu dt\right)^{\frac{1-a}{p}}
L\left(\frac{ap}{p+a-1}, \frac{p}{p+a- 1}\right)^{\frac{p+a-1}{p}}.
\end{eqnarray}
Similarly, we obtain
\begin{eqnarray}\label{eq:estimate_K3}
K_3 &=& \frac{b^2}{2a}\iint_{D_R^c}v^{a + 1}\varphi^{b - 2}\left|\nabla\varphi\right|^2d\mu dt \nonumber\\
&=& \frac{b^2}{2a}\iint_{D_R^c}(h^{\frac{(p-1)(a+1)}{p}}v^{a + 1}\varphi^{\frac{a+1}{p}b})
(h^{-\frac{(p-1)(a+1)}{p}}\varphi^{b - 2-\frac{a+ 1}{p}b}\left|\nabla \varphi\right|^2)d\mu dt \nonumber\\
&\leq&\frac{b^2}{2a}\left(\iint_{D_R^c}h^{p-1}v^p\varphi^b d\mu dt\right)^{\frac{a+1}{p}}\nonumber\\
&&\times\left(\int_0^\infty\int_M h^{-\frac{(p-1)(a+1)}{p-a-1}}\varphi^{[b-2 - \frac{a + 1}{p}b]\frac{p}{p- a - 1}}\left|\nabla\varphi\right|^{\frac{2p}{p-a-1}} d\mu dt\right)^{\frac{p-a-1}{p}} \nonumber \\
&\leq& \frac{C}{a}\left(\iint_{D_R^c} h^{p-1}v^p\varphi^b d\mu dt\right)^{\frac{a + 1}{p}}
J\left(-\frac{ap}{p-a-1}, \frac{2p}{p - a - 1}\right)^{\frac{p-a - 1}{p}},
\end{eqnarray}
and
\begin{eqnarray}\label{eq:estimate_K4}
K_4 &=& - b\iint_{D_R^c}v\varphi^{b - 1}\partial_t\varphi d\mu dt \nonumber\\
&\leq& b\iint_{D_R^c} (h^{\frac{p-1}{p}}v\varphi^{\frac{b}{p}})(h^{-\frac{p-1}{p}}\varphi^{b-1-\frac{b}{p}}\left| \partial_t\varphi \right|) d\mu dt\nonumber\\
&\leq& b
\left(\iint_{D_R^c} h^{p-1}v^p\varphi^b d\mu dt \right)^{\frac{1}{p}}
\left(\int_0^\infty\int_M h^{-1}\varphi ^{[b-1 - \frac{b}{p}]\frac{p}{{p - 1}}}\left|\partial_t\varphi\right|^{\frac{p}{p - 1}}
d\mu dt\right)^{\frac{p - 1}{p}} \nonumber\\
&\leq& C\left(\iint_{D_R^c} h^{p-1}v^p\varphi^b d\mu dt
\right)^{\frac{1}{p}}L\left(0,\frac{p}{p - 1}\right)^{\frac{p - 1}{p}}.
\end{eqnarray}
Substituting (\ref{eq:estimate_K1}), (\ref{eq:estimate_K2}), (\ref{eq:estimate_K3}) and (\ref{eq:estimate_K4}) into (\ref{eq:estimate_of_I_by_four_K}), we obtain
\begin{eqnarray}\label{eq:estimate_I_by_K1234}
I &\leq& \frac{C}{a}\left(\iint_{D_R^c} h^{p-1}v^p\varphi^b d\mu dt\right)^{\frac{1-a}{p}}
J\left(\frac{ap}{p+a-1}, \frac{2p}{p+a
 - 1}\right)^{\frac{p+a - 1}{p}} \nonumber \\
&&+ \frac{C}{1-a}
\left(\iint_{D_R^c}h^{p-1}v^p\varphi^b d\mu dt\right)^{\frac{1-a}{p}}
L\left(\frac{ap}{p+a-1},\frac{p}{p+a - 1}\right)
^{\frac{p+a-1}{p}} \nonumber \\
&&+ \frac{C}{a}
\left(\iint_{D_R^c}h^{p-1}v^p\varphi^b d\mu dt\right)^{\frac{a + 1}{p}}
J\left(-\frac{ap}{p-a-1}, \frac{2p}{p - a - 1}\right)^{\frac{p - a - 1}{p}} \nonumber \\
&&+ C\left(\iint_{D_R^c}h^{p-1}v^p\varphi^b d\mu dt\right)^{\frac{1}{p}}L\left(0,\frac{p}{p - 1}\right)^{\frac{p - 1}{p}},
\end{eqnarray}
which is
\begin{eqnarray}\label{eq:estimate_I_by_K1234_fill_hole}
I &\leq& \frac{C}{a}I^{\frac{1-a}{p}}J\left(\frac{ap}{p+a-1}, \frac{2p}{p-1+a}\right)^{\frac{p-1+a}{p}} +
\frac{C}{1-a}I^{\frac{1-a}{p}}L\left(\frac{ap}{p+a-1}, \frac{p}{p-1+a}\right)^{\frac{p-1+a}{p}} \nonumber \\
&&+\frac{C}{a}I^{\frac{1+a}{p}}
J\left(-\frac{ap}{p-a-1},\frac{2p}{p-1-a}\right)^{\frac{p-1-a}{p}}+CI^{\frac{1}{p}}
L\left(0,\frac{p}{p-1}\right)^{\frac{p-1}{p}}.
\end{eqnarray}
We claim that there exists a constant $C_0>0$ such that
\begin{equation}\label{integral-bound}
\int_{0}^{\infty }\int_{M}h^{p-1}v^{p}d\mu dt \leq C_0 < \infty.
\end{equation}
We divide the proof into two cases:

\textbf{Case 1:} if
\begin{eqnarray*}
\int_0^{\infty}\int_{M}h^{p-1}v^pd\mu dt\leq1,
\end{eqnarray*}
then we let $C_0=1$, and it follow that (\ref{integral-bound}) is  true.

\textbf{Case 2:} If Case 1 is not satisfied, then we obtain
\begin{eqnarray*}
\int_0^{\infty}\int_{M}h^{p-1}v^pd\mu dt>1,
\end{eqnarray*}
Hence, we can find a large enough $R$ such that
\begin{eqnarray}\label{large1}
\iint_{D_R}h^{p-1}v^pd\mu dt>1.
\end{eqnarray}
Recall $p>1$, and choose a positive constant $\beta$ satisfying
$$\frac{{1 + \beta }}{p} < 1.$$
Let $a$ satisfy $0<a\ll\min\{1, \beta\} $. Combining (\ref{eq:estimate_I_by_K1234}) and (\ref{large1}), we obtain
 \begin{eqnarray*}
 I &\leq& CI^{\frac{1 + \beta}{p}}
 \left[ \frac{1}{a }J\left(\frac{ap}{p+a-1},\frac{2p}{p+a - 1}\right)^{\frac{p+a-1}{p}}+L\left(\frac{ap}{p+a-1},
 \frac{p}{p+a-1}\right)^{\frac{p+a-1}{p}}\right.\nonumber\\
  && \left.+ \frac{1}{a }J\left(-\frac{ap}{p-a-1}\frac{2p}{p-a-1}\right)^{\frac{p - a - 1 }{p}}
  + L\left(0,\frac{p}{p - 1}\right)^{\frac{p - 1}{p}} \right].
 \end{eqnarray*}
 It follows that
\begin{eqnarray}\label{eq:estimate_I_by_JQ}
I^{1 - \frac{1 + \beta }{p}} &\leq& C \left[ \frac{1}{a }
J\left(\frac{ap}{p+a-1},\frac{2p}{p+a - 1}
\right)^{\frac{p+a - 1}{p}} + L\left(\frac{ap}{p+a-1},\frac{p}{p+a - 1}\right)^{\frac{p+a - 1}{p}} \right. \nonumber\\
&& + \left. \frac{1}{a }J\left(-\frac{ap}{p-a-1},\frac{2p}{p - a - 1}\right)^{\frac{p - a - 1}{p}} + L\left(0, \frac{p}{p - 1}\right)^{\frac{p - 1}{p}} \right].
\end{eqnarray}
Let $g\in {{C}^{\infty }}[0, \infty )$ be a nonnegative function satisfying
\begin{center}
$g(t)=1$ on $\left[ 0, 1 \right]$;\quad $g(t)=0$ on $\left[ 2, \infty \right)$;\quad $\left| g^{\prime} \right|\leq C_1<\infty$.	
\end{center}
Let $\{\eta_{k} \}_{k\in \mathbb{N}}, \{\gamma_{k}\}_{k\in \mathbb{N}}\in C^{\infty }[0, \infty )$ be two sequences of functions defined respectively by
\begin{eqnarray}\label{eq:definition_eta_k}
\eta _k(t) = g\left(\frac{t}{2^{2k}}\right),
\end{eqnarray}
and
\begin{eqnarray}\label{eq:definition_gamma_k}
\gamma_k(x) = g\left(\frac{r(x)}{2^k}\right),
\end{eqnarray}
where $r(x)=d({x_0}, x)$.

From (\ref{eq:definition_eta_k}) and (\ref{eq:definition_gamma_k}), we have
\begin{eqnarray}\label{eq:partial_t_eta_k}
\left| {\partial_t {{\eta }_{k}}} \right| \left\{
\begin{array}{ll}
\le \frac{C}{{{2}^{2k}}}, \quad t\in [{{2}^{2k}}, {{2}^{2k+1}}], \\
=0, \quad\text{otherwise},
\end{array}\right.
\end{eqnarray}
and
\begin{eqnarray}\label{eq:grad_gamma_k}
\left| \nabla {{\gamma }_{k}} \right| \left\{
\begin{array}{ll}
\leq \frac{C}{2^{k}}, \ x\in B(x_0, 2^{k+1})\setminus B(x_0,2^k), \\
=0, \quad\text{otherwise}.
\end{array}\right.
\end{eqnarray}
Let us define a sequence of functions $\{\varphi_{i}(x,t)\}_{i\in\mathbb{N}}$ by
\begin{eqnarray}\label{eq:test_function_varphi_i}
\varphi_{i}(x,t)=\frac{1}{i}\sum\limits_{k=i+1}^{2i}{{{\eta }_{k}}(t){{\gamma }_{k}}(x)},
\end{eqnarray}
It follows that $\varphi_i(x, t)=1$ when $(x, t)\in \overline{B(x_0, 2^{i})}\times[0, 2^{2i}]$.
Moreover, for distinct $k$, noting that $\text{supp}({\partial_t {{\eta }_{k}}})$ and $\text{supp}(\nabla {{\gamma }_{k}})$ are disjoint respectively,
we obtain for any $\theta>0$
\begin{eqnarray}\label{eq:partial_t_test_function_varphi_i}
\left| {\partial_t {{\varphi }_{i}}} \right|^{\theta }=i^{-\theta }\sum\limits_{k=i+1}^{2i}\left| \gamma_{k}\partial_t (\eta_{k})\right|^{\theta }.
\end{eqnarray}
and
\begin{eqnarray}\label{eq:grad_test_function_varphi_i}
\left| \nabla \varphi_{i} \right|^{\theta }=i^{-\theta }\sum\limits_{k=i+1}^{2i}\left| \eta_{k}\nabla \gamma_{k} \right|^{\theta}.
\end{eqnarray}
Hence
\begin{eqnarray*}
{\varphi _i} \in W_c^{1, 2}(M\times [0, + \infty )).
\end{eqnarray*}
Let
\begin{eqnarray}\label{def-alpha-index}
a=\frac{1}{i}.
\end{eqnarray}
Substituting the above with $\varphi=\varphi_i$ into (\ref{eq:estimate_I_by_JQ}), we obtain
\begin{eqnarray}\label{eq:estimate_I_by_JQ_change_alpha_to_i}
I^{1 - \frac{1+\beta}{p}} &\leq& C\left[iJ\left(\frac{p/i}{p+1/i - 1},\frac{2p}{p+1/i - 1}\right)^{\frac{p+1/i - 1}{p}}
+ L\left(\frac{p/i}{p+1/i - 1},\frac{p}{p+ 1/i -1}\right)^{\frac{ p+1/i -1}{p}}\right. \nonumber\\
&&\left.+ iJ\left(-\frac{p/i}{p-1/i - 1}\frac{2p}{ p- 1/i - 1}\right)^{\frac{p - 1/i - 1}{p}} + L\left(0,\frac{p}{p - 1}\right)^{\frac{p - 1}{p}}\right].
\end{eqnarray}
Substituting \eqref{eq:test_function_varphi_i} into (\ref{eq:definition_IJQ}), and combining (\ref{eq:grad_gamma_k}) and (\ref{eq:grad_test_function_varphi_i}), noting $\eta_k \leq 1$, we obtain
\begin{eqnarray*}
&&iJ\left(\frac{p/i}{p+1/i - 1},\frac{2p}{p+1/i - 1}\right)^{\frac{p+1/i - 1}{p}}
 \nonumber\\
 &=& i\left(\int_0^\infty\int_M h^{\frac{p/i}{p+1/i - 1}}i^{ - \frac{2p}{p +1/i- 1}}
\sum\limits_{k = i + 1}^{2i} \left|\eta_k\nabla\gamma_k \right|^{\frac{2p}{p+1/i - 1}} d\nu dt\right)^{\frac{p + 1/i - 1}{p}} \nonumber\\
&=& i^{ - 1 }\left(\sum\limits_{k = i + 1}^{2i} \int_0^{2^{2k + 1}} \eta_k^{\frac{2p}{p+1/i - 1}}
 \int_{B(x_0, 2^{k + 1})\backslash B(x_0,2^k)} h^{\frac{p/i}{p+1/i - 1}}\left| \nabla \gamma _k \right|^{\frac{2p}{p +1/i- 1}} d\nu dt\right )^{\frac{p + 1/i - 1}{p}} \nonumber\\
&\leq& Ci^{ - 1 }
\left(\sum\limits_{k = i + 1}^{2i} 2^{2k + 1}(2^{k+1})^{\frac{\delta_2p/i}{p+1/i-1}}
2^{ - \frac{2pk}{p +1/i- 1}}
\nu(B(x_0,2^{k + 1})) \right)^{\frac{p + 1/i - 1}{p}},
\end{eqnarray*}
where we have used the condition $(H)$.

Applying volume condition $\nu (B(x_0, r))\leq Cr^{P}\ln ^{Q}r$ with $P = \frac{2}{p - 1}$ and
 $Q = \frac{1}{p - 1}$, we obtain
\begin{eqnarray}
&&iJ\left(\frac{p/i}{p+1/i - 1},\frac{2p}{p+1/i - 1}\right)^{\frac{p+1/i - 1}{p}}
 \nonumber\\
 &\leq& Ci^{ - 1}\left(\sum\limits_{k = i + 1}^{2i} 2^{2k - \frac{2pk}{p + 1/i- 1}}
2^{\frac{k\delta_2p/i}{p+1/i-1}}
2^{kP}k^{Q}\right)^{\frac{p + 1/i - 1}{p}} \nonumber\\
&\leq& Ci^{ - 1 + Q\frac{p + 1/i - 1}{p}}\left(\sum\limits_{k = i + 1}^{2i}
2^{k[2 - \frac{2p}{p +1/i- 1}+\frac{\delta_2p/i}{p+1/i-1} + P]} \right)^{\frac{p + 1/i - 1}{p}} \nonumber\\
&\leq&Ci^{ - 1 + Q\frac{p + 1/i - 1}{p}+\frac{p+1/i-1}{p}}.
\end{eqnarray}
Here we have used that
\begin{eqnarray}\label{eq:limits_sum}
\left(\sum\limits_{k = i + 1}^{2i}
2^{k[2 - \frac{2p}{p +1/i- 1}+\frac{\delta_2p/i}{p+1/i-1} + P]} \right)^{\frac{p + 1/i - 1}{p}}
&=&
2^{2\delta_2p}\left(\sum\limits_{k = i + 1}^{2i}2^{\frac{2kp}{i(p-1)(p +1/i- 1)}}\right)^{\frac{p + 1/i - 1}{p}}
\nonumber\\
&\leq& Ci^{\frac{p + 1/i - 1}{p}}.
\end{eqnarray}
Noting that
$$\limsup\limits_{i \to \infty }
i^{ - 1+ Q\frac{p + 1/i - 1}{p} + \frac{p + 1/i - 1}{p}}=\limsup\limits_{i \to \infty }i^{\frac{1}{i(p-1)}} = 1,$$
we obtain
\begin{eqnarray}\label{eq:estimate_J1}
iJ\left(\frac{p/i}{p+1/i-1},\frac{2p}{p+1/i - 1}\right)^{\frac{p+1/i - 1}{p}} \leq C.
\end{eqnarray}
Substituting (\ref{eq:test_function_varphi_i}) into (\ref{eq:definition_IJQ}), applying (\ref{eq:partial_t_eta_k}) and (\ref{eq:partial_t_test_function_varphi_i}), noting $\gamma_k \leq1$, we obtain
\begin{eqnarray*}
&&L\left(\frac{p/i}{p+1/i-1},\frac{p}{ p+1/i -1}\right)^{\frac{p+ 1/i -1}{p}} \nonumber \\
&=& \left(i^{ - \frac{p}{p+1/i -1}}\sum\limits_{k = i + 1}^{2i}\int_{2^{2k}}^{2^{2k + 1}}
\left|\partial_t \eta_k \right|^{\frac{p}{p+1/i -1}} dt
\int_{B(x_0, 2^{k + 1})}h^{\frac{p/i}{p+1/i-1}}
\gamma_k^{\frac{p}{p+1/i -1}}d\nu \right)^{\frac{p+1/i - 1}{p}} \nonumber\\
&\leq& C\left(i^{ - \frac{p}{p+1/i -1}}\sum\limits_{k = i + 1}^{2i} 2^{ - \frac{2k p}{p+1/i -1}}2^{2k}
(2^{k+1})^{\frac{\delta_2p/i}{p+1/i-1}}
\nu (B(x_0, 2^{k + 1})) \right)^{\frac{p+1/i - 1}{p}} \nonumber\\
& \leq& C\left(i^{- \frac{p}{p+1/i -1}}\sum\limits_{k = i + 1}^{2i}
2^{ - \frac{2kp}{p+1/i -1}}2^{2k}
2^{\frac{k\delta_2p/i}{p+1/i-1}}
2^{kP}k^{Q}\right)^{\frac{p+1/i - 1}{p}} \nonumber \\
&\leq & C^{\prime}\left(i^{ - \frac{p}{p+1/i - 1} + Q}\sum\limits_{k = i + 1}^{2i} 2^{k[2- \frac{2p}{p+1/i - 1} + P]}
\right)^{\frac{p+1/i - 1}{p}},
\end{eqnarray*}
where the term $2^{\frac{k\delta_2p/i}{p+1/i-1}}$ has been absorbed into constant $C^{\prime}$.

Using (\ref{eq:limits_sum}) again, we have
$$
\left(L\left(\frac{p/i}{p+1/i-1},\frac{p}{ p+1/i -1}\right)\right)^{\frac{p+1/i -1}{p}} \leq
 C\left(i^{- \frac{p}{p+1/i - 1} + Q + 1}\right)^{\frac{p+1/i - 1}{p}}.$$
Since
$$\mathop {\limsup }\limits_{i \to \infty }
\left(i^{ - \frac{p}{p+1/i - 1} +Q + 1}
\right)^{\frac{p+1/i - 1}{p}} = 1,$$
we obtain
\begin{eqnarray}\label{eq:estimate_Q1}
\left(L\left(\frac{p/i}{p+1/i-1},\frac{p}{p+ 1/i -1}\right)\right)^{\frac{p+1/i -1}{p}} \leq C.
\end{eqnarray}
Similarly,
\begin{eqnarray}\label{eq:estimate_J2}
&& iJ\left(-\frac{p/i}{p-1/i-1},\frac{2p}{ p- 1/i - 1}\right)^{\frac{p - 1/i - 1}{p}} \nonumber \\
&&= i^{ - 1}\left(\sum\limits_{k = i + 1}^{2i} \int_0^{2^{2k + 1}}
\eta _k ^{\frac{2p}{p - 1/i - 1}}dt \int_{B(x_0,2^{k + 1})\backslash B(x_0,2^k)}
h^{-\frac{p/i}{p-1/i-1}}\left| \nabla \gamma_k \right|^{\frac{2p}{p - 1/i - 1}} d\nu\right)^{\frac{p- 1/i - 1}{p}} \nonumber\\
&&\leq Ci^{ - 1}\left(\sum\limits_{k = i + 1}^{2i}2^{2k + 1}2^{- k\frac{2p}{p - 1/i - 1}}
(2^{k})^{\frac{\delta_1p/i}{p-1/i-1}}\nu (B(x_0,2^{k + 1}))\right)^{\frac{p- 1/i - 1}{p}} \nonumber\\
&&\leq C^{\prime}i^{ - 1 + Q\frac{- 1/i - 1 + p}{p}}
\left(\sum\limits_{k = i + 1}^{2i}2^{k(2 - \frac{2p}{p - 1/i - 1} + P)}\right )^{\frac{p- 1/i - 1}{p}} \nonumber \\
&&\leq Ci^{- 1 + Q\frac{p- 1/i - 1}{p}+ \frac{p- 1/i - 1}{p}} \nonumber \\
&&\leq Ci^{-\frac{1}{i(p-1)}}<\infty.
\end{eqnarray}
and
\begin{eqnarray}\label{eq:estimate_Q2}
\left(L\left(0,\frac{p}{p - 1}\right)\right)^{\frac{p - 1}{p}}
&=& \left(i^{- \frac{p}{p - 1}}\sum\limits_{k = i + 1}^{2i} \int_{2^{2k}}^{2^{2k + 1}}\left|\partial_t \eta_k\right|^{\frac{p}{p - 1}}dt
 \int_{B(x_0, 2^{k + 1})}\gamma_k^{\frac{p}{p - 1}} d\nu\right)^{\frac{p + 1/i - 1}{p}} \nonumber\\
&\leq& \left(i^{ - \frac{p}{p - 1}}\sum\limits_{k = i + 1}^{2i} 2^{ - 2k\frac{p}{p - 1}}2^{2k}\nu (B(x_0,2^{k + 1})) \right)^{\frac{p + 1/i - 1}{p}} \nonumber\\
&\leq& C\left(i^{ - \frac{p}{p - 1} + Q}\sum\limits_{k = i + 1}^{2i}2^{k( - 2\frac{p}{p - 1} + 2 + P)}
\right)^{\frac{p + 1/i - 1}{p}} \nonumber\\
&=& C\left(i^{ - \frac{p}{p - 1} + Q + 1}\right)^{\frac{p + 1/i - 1}{p}} \nonumber\\
&=& C<\infty .\nonumber\\
\end{eqnarray}
Combining (\ref{eq:estimate_J1}), (\ref{eq:estimate_Q1}, (\ref{eq:estimate_J2}), (\ref{eq:estimate_Q2}) with (\ref{eq:estimate_I_by_JQ_change_alpha_to_i}), we have
\begin{eqnarray}
\int_0^{2^{2i}} \int_{B(x_0, 2^{i})}h^{p-1}v^p d\mu dt\leq C < \infty.
\end{eqnarray}
It follows by letting $i \rightarrow \infty$ that
\begin{eqnarray*}\label{eq:Int_u_le_C}
\int_{0}^{\infty }\int_{M}h^{p-1}v^{p}d\mu dt \leq C < \infty.
\end{eqnarray*}
Hence, the claim (\ref{integral-bound}) is true.

Substituting $\varphi = \varphi_i$ and $R = 2^i$ into (\ref{eq:estimate_I_by_K1234}),
combining with (\ref{eq:estimate_J1}), (\ref{eq:estimate_Q1}, (\ref{eq:estimate_J2}), (\ref{eq:estimate_Q2}) and (\ref{integral-bound}),
repeating the same procedures in (\ref{eq:estimate_I_by_K1234}), we obtain
\begin{eqnarray}\label{eq:the_last_inequality}
\int_{0}^{2^{2i} }\int_{B_{2^i}}h^{p-1}v^{p}d\mu dt &\leq& C\left\{
\left(\iint_{D_{2^i}^c} h^{p-1}v^p d\mu dt \right)^{\frac{1-\frac{1}{i}}{p}}
+\left(\iint_{D_{2^i}^c} h^{p-1}v^p d\mu dt \right)^{\frac{1+\frac{1}{i}}{p}}\right.\nonumber\\
&&\left.+\left(\iint_{D_{2^i}^c} h^{p-1}v^p d\mu dt \right)^{\frac{1}{p}}\right\}.
\end{eqnarray}
Letting $i \rightarrow \infty$, from (\ref{integral-bound}), we have
$$\int_{0}^{\infty }\int_{M} h^{p-1}v^p d\mu dt = 0,$$
which implies
$$v\equiv 0. $$
Noting that $u=hv$, hence $u\equiv0$. However, the above leads to the contradiction with the positiveness of $u$. Hence, there
exists no global positive solution to problem (\ref{eq-pra}).
\end{proof}

\section{Global existence of positive solution}\label{GE}
In this section, we show the sharpness of $P, Q$ in Theorem \ref{thm:3}.
It suffices to show that $Q$ in (\ref{volume-1}) can not be relaxed.

%, which can be found in many papers, for example.
\begin{proof}[Proof of Theorem \ref{thm:3}]\RM
Define the operator
\begin{eqnarray}\label{eq: Tu}
Tv(x,t)=\int_M \tilde{P}_t(x,y)v_0(y)d\mu(y) + \int_0^t \int_M \tilde{P}_{t-s}(x,y)h^{p-1}v^p(y,s) d\mu(y)ds.
\end{eqnarray}
acting on the following space
\begin{eqnarray}\label{assumption: u}
S_M=\left\{v\in L^{\infty}(M\times[0,\infty))|\ 0 \leq v(x,t) \leq \lambda \tilde{P}_{t+\delta}(x,x_0)\right\}.
\end{eqnarray}
where  $\lambda >0$ is a constant to be chosen later, and $\delta>1$ is a large fixed constant.
It follows that $S_M$ is a closed set of $L^{\infty}(M\times[0,\infty), d\mu)$.

Let $v_0$ satisfy
\begin{equation}\label{assumption: u_0}
0\leq v_0(x)\leq\frac{\lambda}{2}\tilde{P}_{\delta}(x,x_0).
\end{equation}
Now let us show $TS_M\subset S_{M}$.

From (\ref{assumption: u_0}), and applying (\ref{hk-semigroup}), we have
\begin{eqnarray}\label{eq: Tu_part_1}
\int_{M}\tilde{P}_t(x,y)v_0(y)d\mu(y)&\leq&\frac{\lambda}{2}\int_{M}\tilde{P}_t(x,y)
\tilde{P}_{\delta}(y,x_0)d\mu(y)\nonumber\\
&=&\frac{\lambda}{2}\tilde{P}_{t+\delta}(x,x_0).
\end{eqnarray}
From $(DUE)$ and (\ref{assumption: u}), we have
\begin{eqnarray}\label{eq: en_taro_adun}
&&\int_0^t \int_M \tilde{P}_{t-s}(x,y)h^{p-1}v^p(y,s) d\mu(y)ds \nonumber\\
&\leq& C_1\lambda^p\int_0^t \int_M \tilde{P}_{t-s}(x,y)\tilde{P}_{s+\delta}^p(y,x_0) d\mu(y)ds \nonumber\\
&\leq& C_2\lambda^p\int_0^t \frac{1}{\mu(B(x_0,\sqrt{s+\delta}))^{p-1}} ds\int_M \tilde{P}_{t-s}(x,y)\tilde{P}_{s+\delta}(y,x_0) d\mu(y) \nonumber\\
&\leq& C_3\lambda^p\tilde{P}_{t+\delta}(x,x_0) \int_0^t \frac{1}{\mu(B(x_0,\sqrt{s+\delta}))^{p-1}} ds,
\end{eqnarray}
where we have used that $h\asymp1$.

Recalling that for large enough $r$,
$$\mu_0(B(x_0,r)) \geq  c_2r^{P}\ln^{Q+\varepsilon} r,$$
and since $h\asymp1$, and $d\mu=h^2d\mu_0$, we have
\begin{eqnarray*}
\mu(B(x_0,r)) \geq  c_3r^{P}\ln^{Q+\varepsilon} r.
\end{eqnarray*}
When $\delta$ is large enough, we obtain
\begin{eqnarray}\label{eq: int_VB_leq_C3}
&&\int_0^t \frac{1}{\mu(B(x_0,\sqrt{s+\delta}))^{p-1}} ds \nonumber\\
&\leq&\int_0^t\frac{1}{\left[C_1(s+\delta)^{\frac{P}{2}}(\ln\sqrt{s+\delta} )^{Q+\varepsilon}\right]^{p-1}}ds\nonumber\\
&\leq& C_42^{1+\varepsilon(p-1)}\int_0^t \frac{1}{(s+\delta)[\ln(s+\delta)]^{1+\varepsilon(p-1)}} ds \nonumber\\
&\leq& C_42^{1+\varepsilon(p-1)}\int_0^\infty \frac{1}{(s+\delta)[\ln(s+\delta)]^{1+\varepsilon(p-1)}} ds \nonumber\\
&\leq& C_5 < \infty,
\end{eqnarray}
where we have used that $P=\frac{2}{p-1}$, $Q=\frac{1}{p-1}$.

Combining (\ref{eq: en_taro_adun}) with (\ref{eq: int_VB_leq_C3}), we obtain, for small enough $\lambda$,
\begin{eqnarray}\label{eq: Tu_part_2}
\int_0^t \int_M \tilde{P}_{t-s}(x,y)h^{p-1}v^p(y,s) d\mu(y)ds &\leq& C_5C_3\lambda^p
\tilde{P}_{t+\delta}(x,x_0) \nonumber\\
&\leq& \frac{\lambda}{2}\tilde{P}_{t+\delta}(x,x_0).
\end{eqnarray}
Combining (\ref{eq: Tu}),(\ref{eq: Tu_part_1}) with (\ref{eq: Tu_part_2}), we obtain
$$0 \leq Tv \leq \lambda \tilde{P}_{t+\delta}(x,x_0).$$
Hence
$$TS_{M} \subset S_M.$$
Now we show that $T$ is a contraction map. For $v_1$, $v_2\in S_M$, we have
\begin{eqnarray}\label{eq: en_taro_artanis}
\left|Tv_1(x,t)-Tv_2(x,t)\right| \leq
\int_0^t \int_M\tilde{P}_{t-s}(x,y)h^{p-1}\left|v_1^p(y,s)-v_2^p(y,s)\right| d\mu(y)ds.
\end{eqnarray}
Noting that
$$\left|v_1^p(y,s)-v_2^p(y,s)\right| \leq p \max\{v_1^{p-1}(y,s),v_2^{p-1}(y,s)\}\left|v_1(y,s)-v_2(y,s)\right|,$$
and combining with $(DUE)$, (\ref{assumption: u}) and (\ref{eq: int_VB_leq_C3}),
  and using that $h\asymp1$, we obtain from (\ref{eq: en_taro_artanis})
that
\begin{eqnarray*}
&& \left|Tv_1(x,t)-Tv_2(x,t)\right| \nonumber\\
&\leq& C_6p\lambda^{p-1}\left\| v_1-v_2 \right\|_{L^\infty}\int_0^t \int_M \tilde{P}_{t-s}(x,y)\tilde{P}_{s+\delta}^{p-1}(y,x_0) d\mu(y)ds \nonumber\\
&\leq& C_7p\lambda^{p-1}\left\| v_1-v_2 \right\|_{L^\infty}
\int_0^t \frac{1}{\mu(B(x_0,\sqrt{s+\delta}))^{p-1}}ds
\int_M \tilde{P}_{t-s}(x,y) d\mu(y)  \nonumber\\
&\leq& C_8p\lambda^{p-1}C_1^{p-1}\left\| v_1-v_2 \right\|_{L^\infty}
\int_0^t \frac{1}{\mu(B(x_0,\sqrt{s+\delta}))^{p-1}} ds \nonumber\\
&\leq& C_9p\lambda^{p-1}\left\| v_1-v_2 \right\|_{L^\infty},
\end{eqnarray*}
where we have used that (\ref{hk-markov}) and (\ref{eq: int_VB_leq_C3}).

Choosing $\lambda$ small enough so that $C_9p\lambda^{p-1}<1$, we obtain that $T$ is a contraction map.
Applying fixed point theorem, we know
there exists a fixed point $v\in S_M$ satisfying
\begin{equation}\label{fix}
v(x,t)=\int_M \tilde{P}_t(x,y)v_0(y)d\mu(y) + \int_0^t \int_M
\tilde{P}_{t-s}(x,y)h^{p-1}(y)v^p(y,s) d\mu(y)ds.
\end{equation}
Since  $v_0\gneqq0$, then $v$ is positive on $M$.
 Since $ v_0, v\in L^2(M, d\mu)$,
 by \cite[Theorem 7.6 and 7.7]{Grigoryan-book},
we know the integrals in (\ref{fix}) are both smooth on $M\times(0,\infty)$, hence
we obtain that $v$ is a global positive solution of problem (\ref{eqT}), Furthermore, $u=hv$ is a global positive solution of problem (\ref{eq-pra}).
\end{proof}

%\begin{remark}\RM
%From the proof, one can find that (\ref{vol-e-1}) in Theorem \ref{thm:3} can be replaced by another condition
%\begin{equation}\label{vol-e-3}
%\int^{\infty}\frac{dr}{\mu(B(x_0, \sqrt{r}))^{p-1}}<\infty.
%\end{equation}
%Here the notation $\int^{\infty}$ means we take the integral near infinity.
%\end{remark}

%We do not know whether (\ref{vol-e-3}) is also sharp for Theorem \ref{thm:1}. We conjecture
%\begin{conjecture}\RM
%If, for some reference point ${x_0}\in M$, the following
%\begin{eqnarray*}
%\int^{\infty}\frac{dr}{\mu(B(x_0, \sqrt{r}))^{p-1}}=\infty.
%\end{eqnarray*}
%holds,
%then  problem (\ref{eq-pra}) admits no global positive solution.
%\end{conjecture}

\begin{acknowledgment}
\RM
The authors would like to express their deep gratitude to Prof.
Qi S. Zhang from University of California Riverside who initiated the study of the above problems, and  bringing our attentions to his paper
\cite{Zhang-duk99}. The authors would also like to thank Prof. Verbitsky from University of Missouri for helpful communications in Section \ref{example}.
\end{acknowledgment}

\end{document}